\documentclass[11pt]{article}
\usepackage[utf8]{inputenc}

\usepackage{amsmath,amssymb,amsthm} 
\usepackage{xcolor}
\usepackage[numbers,sort]{natbib}
\global\long\def\argmin{\operatorname*{argmin}}%

\usepackage[T1]{fontenc}
\usepackage{geometry}
\usepackage{amsmath}
\usepackage{amsthm}
\usepackage{amssymb}
\usepackage{array}
\usepackage{float}
\usepackage{booktabs}
\usepackage{tablefootnote}
\usepackage{caption}
\usepackage[title]{appendix}
\usepackage[colorlinks=true, linkcolor=blue, citecolor=blue, urlcolor=blue, anchorcolor=blue]{hyperref}
\allowdisplaybreaks

\definecolor{green}{rgb}{0,0.6,0.0}
\usepackage{amsmath}
\usepackage{amsfonts}
\usepackage{latexsym}
\usepackage{amssymb}

\newtheorem{thm}{Theorem}[section]
\newtheorem{theorem}[thm]{Theorem}

\newtheorem{lemma}[thm]{Lemma}

\newtheorem{proposition}[thm]{Proposition}

\newtheorem{remark}[thm]{Remark}

\newcommand{\beq}{\begin{equation}}
\newcommand{\eeq}{\end{equation}}
\newcommand{\beqa}{\begin{eqnarray}}
\newcommand{\eeqa}{\end{eqnarray}}
\newcommand{\beqas}{\begin{eqnarray*}}
\newcommand{\eeqas}{\end{eqnarray*}}
\newcommand{\ei}{\end{itemize}}

\newcommand{\vgap}{\vspace{.1in}}

\newcommand{\R}{\mathbb{R}}

\newcommand{\lam}{{\lambda}}

\newcommand{\inner}[2]{\langle #1,#2\rangle}

\newcommand{\dom}{\mathrm{dom}\,}

\newcommand{\tx}{\tilde x}

\newtheorem{remar}{Remark}[section]

\global\long\def\tx{\tilde{x}}%
\global\long\def\rn{\Re^{n}}%
\global\long\def\cal{\mathcal}%
\global\long\def\R{\Re}%
\global\long\def\r{\Re}%
\global\long\def\lam{\lambda}%
\global\long\def\argmin{\operatorname*{argmin}}%
\global\long\def\dom{\operatorname*{dom}}%
\global\long\def\inner#1#2{\langle#1,#2\rangle}%
\global\long\def\cConv{\overline{{\rm Conv}}\ }%
\usepackage{titlesec}
\titleformat{\section}
  {\normalfont\fontsize{14}{15}\bfseries}{\thesection}{1em}{}
  \titleformat{\subsection}
  {\normalfont\fontsize{12}{15}\bfseries}{\thesubsection}{1em}{}

\title{Efficient parameter-free restarted accelerated gradient methods for convex and strongly convex optimization}

\author{Arnesh Sujanani  \thanks{Department of Combinatorics and Optimization, University of Waterloo, Waterloo, ON, N2L 3G1. (Email: {\tt a3sujana@uwaterloo.ca}). This work was primarily done while this author was a PhD student at Georgia Institute of Technology. While at Georgia Institute of Technology, he was partially supported by AFORS Grant FA9550-22-1-0088.} \and Renato D.C. Monteiro \thanks{Stewart School of Industrial and Systems Engineering, Georgia Institute of Technology, Atlanta, GA, 30332-0205. (Email: {\tt monteiro@isye.gatech.edu}). This author was partially supported by AFORS Grant FA9550-22-1-0088.} }
\date{October 5, 2024 (revised: October 11, 2024)
}

\begin{document}
\maketitle
\begin{abstract}
This paper develops a new parameter-free restarted method, namely RPF-SFISTA, and a new parameter-free aggressive regularization method, namely A-REG, for solving strongly convex and convex composite optimization problems, respectively. RPF-SFISTA has the major advantage that it requires no knowledge of both the strong convexity parameter of the entire composite objective and the Lipschitz constant of the gradient. Unlike several other restarted first-order methods which restart an accelerated composite gradient (ACG) method after a predetermined number of ACG iterations have been performed, RPF-SFISTA checks a key inequality at each of iterations to determine when to restart. Extensive computational experiments show that RPF-SFISTA is roughly 3 to 15 times faster than other state-of-the-art restarted methods on four important classes of problems. The A-REG method, developed for convex composite optimization, solves each of its strongly convex regularized subproblems according to a stationarity criterion by using the RPF-SFISTA method with a possibly aggressive choice of initial strong convexity estimate. This scheme is thus more aggressive than several other regularization methods which solve their subproblems by running a standard ACG method for a predetermined number of iterations. 
\end{abstract}
\section{Introduction}\label{sec:intro}
This paper presents
present new restarted and aggressive parameter-free first-order methods
for solving
the composite optimization  problem
\begin{equation} \label{OptProb1}
\phi_{*}:=\min \{\phi(z) :=  f(z) + h(z) \},
\end{equation}
where $\phi$ is a function that is assumed to be either convex or $\bar \mu$ strongly convex, $h:\Re^n \to (-\infty,\infty]$ is a closed proper convex function, and
$f:\Re^n \to \Re$ is a real-valued differentiable
convex function whose
gradient is $\bar L$--Lipschitz continuous. 

The main focus of this paper is to propose a computationally efficient restarted parameter-free method, namely RPF-SFISTA, for solving strongly convex composite optimization (SCCO) problems. At each iteration, RPF-SFISTA proceeds as follows: it first calls a strongly convex accelerated composite gradient method (S-ACG) developed in \cite{SujananiMonteiro} with an (usually aggressive) estimate $\mu$ of the strong convexity parameter $\bar \mu$ until either a desired stationary point is obtained or a certain key inequality is violated; in the latter case S-ACG is invoked again with strong convexity estimate set to $\mu/2$ and initial point set to the best point obtained in the previous S-ACG call. Using this scheme, RPF-SFISTA restarts a logarithmic number of times and it achieves an iteration complexity of $\tilde{\mathcal O}\left(\sqrt{\bar L/\bar \mu}\right)$. RPF-SFISTA's superior numerical performance is also displayed through extensive computational experiments.

This paper also proposes a dynamic aggressive regularization method, namely A-REG, for solving convex composite optimization (CCO) problems. A-REG 
solves a sequence of strongly convex regularized subproblems using RPF-SFISTA with an aggressive choice of strong convexity estimate. Under the assumption of bounded sublevel sets, A-REG achieves a complexity of $\tilde{\mathcal O}\left(\sqrt{\bar L/\epsilon}\right)$, which is optimal up to logarithmic terms.

{\bf Literature review.} We divide our discussion here into methods that were designed to solve SCCO and CCO problems, respectively. 

\textit{Methods for SCCO}: 
Methods that restart accelerated composite gradient (ACG) methods for solving SCCO problems have been proposed as early as 2008 in \cite{LanMRest} (see also \cite{LanRen2013PenMet} for the published version) while methods that restart strongly convex variants of accelerated composite gradient (S-ACG) methods have been proposed as early as 2013 in \cite{nesterov2012gradient}. Table \ref{ACG Methods Comparison} (resp. Table~\ref{S-ACG Restarted Methods}) below summarizes the differences between the existing restarted ACG (resp. S-ACG) methods.

We now briefly comment on each of the columns in both tables. The ``Universality'' column describes which parameters each method is universal with respect to. Methods that consider the setting where $f$ is $\bar \mu_f$-strongly convex (resp. $\bar L$-smooth) and require no knowledge of $\bar \mu_f$ (resp. $\bar L$) are said to be $\bar \mu_f$-universal (resp. $\bar L$-universal). 
Likewise, methods that consider the setting where the entire composite function $\phi$ is 
assumed to be $\bar \mu$-strongly convex and 
that require no knowledge of $\bar \mu$ are said to be $\bar \mu$-universal. If an entry is marked with $*$, then it means that the method is not universal with respect to that 
parameter. The ``Composite Objective'' column displays whether the method considers a  
general composite objective as in \eqref{OptProb1} or the special case where 
$h$ is the indicator function of a closed convex set (i.e., $h=\delta_{C})$. The ``Stationarity'' column indicates whether the 
method terminates according to a checkable stationarity termination criterion or 
according to a condition that involves the optimal value of \eqref{OptProb1} (which is 
usually unknown). The ``Convergence Proof'' column presents whether a method has any convergence guarantees or not. Finally, the 
``Restart Condition'' column indicates whether the ACG (or S-ACG) restarts when a checkable condition is satisfied at some iteration or after a predetermined number of iterations is performed.

It has been observed in practice that restarted S-ACG methods tend to perform much better than restarted ACG methods. Also, methods that restart based on a checkable condition tend to have better computational performance than methods that restart based on a predetermined number of iterations.
\begin{table}[H]
\captionsetup{font=scriptsize}
\begin{centering}
\begin{tabular}{ccccccc}
\toprule 
\textbf{\tiny{}Name} & \textbf{\tiny{}Universality} & \textbf{\tiny{}Composite Objective} & \textbf{\tiny{}Stationarity} & \textbf{\tiny{}Convergence Proof}  & \textbf{\tiny{}Restart Condition} \tabularnewline

% \textbf{\tiny{}} & \textbf{\tiny{or $\bar \mu_f$-Universal}} & \textbf{\tiny{ACG}} & \textbf{\tiny{}of ACG} & {\tiny{\textbf{Proof}}}  & \textbf{\tiny{Free}} &  \textbf{\tiny{Objective}}\tabularnewline
\midrule 

{\scriptsize{\cite{LanRen2013PenMet, LanMRest}}} & {\scriptsize{}$(*,*)$} & {\scriptsize{}No ($h=\delta_{C}$)} & {\scriptsize{}No} & {\scriptsize{}Yes} & {\scriptsize{}Predetermined}\tabularnewline

{\scriptsize{Sync||FOM (2022) \cite{Renegar}}} & {\scriptsize{}$(\bar \mu_f,\bar L)$} & {\scriptsize{}No ($h=\delta_{C}$)} & {\scriptsize{}No} & {\scriptsize{}Yes} & {\scriptsize{}Checkable}\tabularnewline

{\scriptsize{\cite{Alamo1, Alamo2, Alamo3, Fercoq, Aujol, Roulet}}} & {\scriptsize{}$(\bar \mu, *)$} & {\scriptsize{}Yes} & {\scriptsize{}Yes} & {\scriptsize{}Yes} & {\scriptsize{}Predetermined}\tabularnewline

{\scriptsize{}Free-FISTA (2023) \cite{Aujol2}} & {\scriptsize{$(\bar \mu,\bar L)$}} & {\scriptsize{}Yes} & {\scriptsize{}Yes} & {\scriptsize{}Yes} & {\scriptsize{}Predetermined} \tabularnewline

\bottomrule
\end{tabular}
\par\end{centering}
\caption{\scriptsize{Comparison of restarted ACG methods for SCCO.} }
\label{ACG Methods Comparison}
\end{table}

% after performing a predetermined number of iterations

% a or a key condition is satisfied at some iteration.
 
This paragraph solely focuses on describing strongly convex variants of FISTA or S-ACG methods that have been previously developed. Methods for SCCO problems that use S-ACG and that are parameter-dependent and require knowledge of the strong convexity parameter underlying the objective function have been proposed in the following works \cite{Monteiro2016, YheMoneiroNash, ChambollePock, Necorora, Chambolle, florea2018accelerated}. Universal restarted methods, including the RPF-SFISTA method in this paper, have also been proposed (see also \cite{LinQihang, GreedyFISTA, OptPar, nesterov2012gradient}). Table \ref{S-ACG Restarted Methods} below highlights the differences between each of the restarted S-ACG methods.
\begin{table}[H]
\captionsetup{font=scriptsize}
\begin{centering}
\begin{tabular}{cccccccc}
\toprule 
\textbf{\tiny{}Name} & \textbf{\tiny{}Universality} & \textbf{\tiny{}Composite Objective} & \textbf{\tiny{}Stationarity}  & \textbf{\tiny{}Convergence Proof} & \textbf{\tiny{}Restart Condition} \tabularnewline

% \textbf{\tiny{}} & \textbf{\tiny{or $\bar \mu_f$-Universal}} & \textbf{\tiny{ACG}} & \textbf{\tiny{}of ACG} & {\tiny{\textbf{Proof}}}  & \textbf{\tiny{Free}} &  \textbf{\tiny{Objective}}\tabularnewline
\midrule

{\scriptsize{\cite{nesterov2012gradient, LinQihang}}} & {\scriptsize{}$(\bar \mu_f,\bar L)$} & {\scriptsize{}Yes} & {\scriptsize{}Yes} & {\scriptsize{}Yes} & {\scriptsize{}Checkable} \tabularnewline

{\scriptsize{GR-FISTA (2022) \cite{GreedyFISTA}}} & {\scriptsize{}$(\bar \mu_f,*)$} & {\scriptsize{}Yes} & {\scriptsize{}No} & {\scriptsize{}No} & {\scriptsize{}Checkable} \tabularnewline

{\scriptsize{}SCAR (2023) \cite{OptPar}} &{\scriptsize{}$(\bar \mu_f,\bar L)$} & {\scriptsize{}No} & {\scriptsize{}Yes} & {\scriptsize{}Yes} & {\scriptsize{}Checkable} \tabularnewline

\textbf{\scriptsize{}RPF-SFISTA
}
& {\scriptsize{}$(\bar \mu,\bar L)$} & {\scriptsize{}Yes} & {\scriptsize{}Yes} & {\scriptsize{}Yes} & {\scriptsize{}Checkable} \tabularnewline
\bottomrule
\end{tabular}
\par\end{centering}
\caption{\scriptsize{Comparison of RPF-SFISTA with other restarted S-ACG methods for SCCO.} }
\label{S-ACG Restarted Methods}
\end{table}

As seen from Table \ref{S-ACG Restarted Methods}, RPF-SFISTA is the only parameter-free $\bar \mu$-universal method that restarts a S-ACG method based on a checkable condition. It is observed in practice that methods that restart a S-ACG method based on a checkable condition tend to perform better in practice. We display through extensive computational experiments that RPF-SFISTA is roughly three to fifteen times faster than other state-of-the-art restarted methods on four important classes of problems. It is also worth mentioning that $\bar \mu$ can be much larger than $\bar \mu_f+\bar \mu_h$.  Hence, methods that are $\bar \mu$-universal have the advantage that their complexities tend to be better than those of $\bar \mu_f$ and $\bar \mu_h$-universal methods and nonuniversal methods which depend on $\bar \mu_f+\bar \mu_h$.

\textit{Methods for CCO}: Many direct methods based on proximal gradient, ACG, or FISTA methods have been developed to solve CCO problems (see \cite{nesterov2012gradient, nesterov1983method, beck2009fast, Scheinberg, Latafat, Malitsky, Li, AvgCurv, ACFISTA}). Heuristic restart schemes that restart FISTA and that achieve good practical performance were also further proposed by O'Donoghue and Candès in \cite{Candes}. Convergence proofs of their generic heuristic schemes have been provided only in very special cases \cite{Moursi}.

In another line of research, regularization methods for solving CCO problems have been proposed as early as 2012 by Nesterov in \cite{NesterovNEWs}. In the unconstrained setting, Nesterov proposed a static regularization method where he applied a S-ACG method a single time to the regularized objective $\phi(z)+\delta\|z-z_0\|^2$. His method achieves a complexity of $\mathcal O\left({\sqrt{\bar L/\epsilon}}\right)$ for finding an $\epsilon$-stationary point, a complexity which has not yet been established by a direct method. Similar regularization schemes that are dynamic and that consider the more general composite setting have been proposed in \cite{CatalystAgain, CatalystNC, WJRproxmet1, Aaronetal2017}. Roughly speaking, dynamic regularization methods solve a sequence of regularized subproblems of the form $\phi(z)+\delta_k\|z-z_k\|^2$. If the solution $z_{k+1}$ of the $k$-th subproblem is not optimal for \eqref{OptProb1}, the regularization factor $\delta_{k}$ is halved, the next prox center is updated to be $z_{k+1}$, and S-ACG is invoked again to solve the new subproblem. More recently, dynamic parameter-free regularization methods, including the A-REG method proposed in this paper and the methods in \cite{CurvatureFree, OptPar}, have also been proposed for solving CCO problems. Table~\ref{tab:summary_tbl} compares the features of A-REG with other regularization methods  for CCO. 

We briefly describe each of its columns. The first three columns describe the same features as the first three columns of Table \ref{ACG Methods Comparison}, whose formal descriptions were given previously. The fourth column presents whether a method solves its strongly convex subproblems according to a stationarity condition that the ACG variant, which the method uses, checks at each of its iterations or whether the method solves each subproblem by running a predetermined number of ACG iterations. The fifth column displays whether a method uses a restarted S-ACG method with an usually aggressive strong convexity estimate or a standard S-ACG method to solve its regularized subproblems. The last column presents whether a method is a static regularization method or a dynamic one that adaptively updates its prox-center and regularization parameter.

\begin{table}[H]
\captionsetup{font=scriptsize}
\begin{centering}
\begin{tabular}{cccccccc}
\toprule 
\textbf{\tiny{}Name} & \textbf{\tiny{}Universality} & \textbf{\tiny{}Composite} & {\tiny{\textbf{Stationarity}}}  & \textbf{\tiny{}Checkable Condition} & \textbf{\tiny{}Restart} & \textbf{\tiny{Dynamic}}\tabularnewline

\textbf{\tiny{}} & \textbf{\tiny{of Method}} & \textbf{\tiny{Objective}} & {\tiny{\textbf{Termination}}} & \textbf{\tiny{}for Subproblem} & \textbf{\tiny{S-ACG}} & \textbf{\tiny{Regularization}}\tabularnewline
\midrule 

{\scriptsize{}Nesterov (2012) \cite{NesterovNEWs}} & {\scriptsize{}*} & {\scriptsize{}No} & {\scriptsize{}Yes} & {\scriptsize{}Yes} & {\scriptsize{}No} & {\scriptsize{}No}\tabularnewline

{\scriptsize{}Catalyst (2015) \cite{CatalystAgain}} & {\scriptsize{}*} & {\scriptsize{}Yes} & {\scriptsize{}No} & {\scriptsize{}No} & {\scriptsize{}No} & {\scriptsize{}Yes}\tabularnewline

{\scriptsize{}\cite{WJRproxmet1, Aaronetal2017}} & {\scriptsize{}*} & {\scriptsize{Yes}} & {\scriptsize{}Yes} & {\scriptsize{}Yes} & {\scriptsize{}No} & {\scriptsize{}Yes}\tabularnewline

{\scriptsize{}4WD-Catalyst (2018) \cite{CatalystNC}} & {\scriptsize{}*} & {\scriptsize{Yes}} & {\scriptsize{}Yes} & {\scriptsize{}No} & {\scriptsize{}No} & {\scriptsize{}Yes}\tabularnewline

{\scriptsize{}APD Method (2022) \cite{CurvatureFree}} & {\scriptsize{}$\bar L$} & {\scriptsize{Yes}} & {\scriptsize{}Yes} & {\scriptsize{}Yes} & {\scriptsize{}No} & {\scriptsize{}Yes}\tabularnewline

{\scriptsize{}AR Method (2023) \cite{OptPar} \tablefootnote{In the general composite setting, the AR method in \cite{OptPar} is not developed or presented as a parameter-free method and requires knowledge of the Lipschitz constant, $\bar L$. However, in the unconstrained setting the authors develop a parameter-free variant of the AR method, which they say can be extended to the general composite setting.}} & {\scriptsize{}$\bar L$} &{\scriptsize{}Yes} & {\scriptsize{}Yes} & {\scriptsize{}No} & {\scriptsize{}No} & {\scriptsize{}Yes}\tabularnewline

\textbf{\scriptsize{}A-REG
}
& {\scriptsize{}$\bar L$} & {\scriptsize{}Yes} & {\scriptsize{}Yes} & {\scriptsize{}Yes} & {\scriptsize{}Yes} & {\scriptsize{}Yes}\tabularnewline
\bottomrule
\end{tabular}
\par\end{centering}
\caption{\scriptsize{Comparison of A-REG with other regularization methods for CCO.} }
\label{tab:summary_tbl}
\end{table}

As can be seen from Table 3, A-REG is the only method that solves its dynamic regularization subproblems using a restart S-ACG method which makes an adaptive choice of strong convexity parameter.

% invokes a restart S-ACG to solve its dynamic regularization subproblems.

% Table~\ref{tab:summary_tbl} displays that A-REG, to the best of our knowledge, is the first parameter-free dynamic regularization method that uses a restarted S-ACG method with an usually aggressive strong convexity estimate for solving its strongly convex regularized subproblems. Hence, A-REG, is more aggressive than existing methods which all set the strong convexity parameter to be the one associated with the regularization term.

% have the same meaning as columns 3-5 of table 1, whose descriptions were given precisely in the previous paragraph.

{\bf Organization of the paper.} Subsection~\ref{subsec:notation} presents basic definitions and notations used throughout this paper. Section~\ref{sec:SCCO} formally describes the RPF-SFISTA method for solving SCCO problems and its main complexity result and analysis. Section~\ref{CCO Optimization} presents the A-REG method for solving CCO problems and its main complexity result and analysis. Section~\ref{sec:numerical} presents extensive computational experiments that display the superior numerical performance of RPF-SFISTA compared to other state-of-the-art restart schemes on four different important classes of composite optimization problems. Finally, Appendix~\ref{SC-FISTA Appendix Proof} is dedicated to proving an important proposition used in the complexity analysis of RPF-SFISTA.

\subsection{Basic Definitions and Notations} \label{subsec:notation}

This subsection presents notation and basic definitions used in this paper.
	
Let $\Re_{+}$  and $\Re_{++}$ denote the
set of nonnegative and positive
real numbers, respectively. We denote by  $\Re^n$  an $n$-dimensional inner product space with inner product and associated norm denoted by $\inner{\cdot}{\cdot}$ and $\|\cdot\|$, respectively. We use $\Re^{l\times n}$ to denote the set of all $l\times n$ matrices and ${\mathbb S}_n^+$ to denote the set of positive semidefinite matrices in $\r^{n\times n}$.  The smallest positive singular value of a nonzero linear operator $Q:\Re^n\to \Re^l$ is denoted by $\nu^+_Q$. For a given closed convex set $Z \subset \Re^n$, its boundary is denoted by $\partial Z$ and the distance of a point $z \in \Re^n$  to $Z$ is denoted by ${\rm dist}(z,Z)$. The indicator function of $Z$, denoted by $\delta_Z$, is defined by $\delta_Z(z)=0$ if $z\in Z$, and $\delta_Z(z)=+\infty$ otherwise.
 For any $t>0$ and $b\geq 0$, we let $\log_b^+(t):=\max\{\log t, b\}$, and we  define ${\cal O}_1(\cdot) = {\cal O}(1 + \cdot)$.

The domain of a function $h :\Re^n\to (-\infty,\infty]$ is the set $\dom h := \{x\in \Re^n : h(x) < +\infty\}$.
Moreover, $h$ is said to be proper if
$\dom h \ne \emptyset$. The set of all lower semi-continuous proper convex functions defined in $\Re^n$ is denoted by $\cConv \rn$. The $\varepsilon$-subdifferential of a proper function $h :\Re^n\to (-\infty,\infty]$ is defined by 
\begin{equation}\label{def:epsSubdiff}
\partial_\varepsilon h(z):=\{u\in \Re^n: h(z')\geq h(z)+\inner{u}{z'-z}-\varepsilon, \quad \forall z' \in \Re^n\}
\end{equation}
for every $z\in \Re^n$.	The classical subdifferential, denoted by $\partial h(\cdot)$,  corresponds to $\partial_0 h(\cdot)$.  
Recall that, for a given $\varepsilon\geq 0$, the $\varepsilon$-normal cone of a closed convex set $C$ at $z\in C$, denoted by  $N^{\varepsilon}_C(z)$, is 
$$N^{\varepsilon}_C(z):=\{\xi \in \Re^n: \inner{\xi}{u-z}\leq \varepsilon,\quad \forall u\in C\}.$$
The normal cone of a closed convex set $C$ at $z\in C$ is denoted by  $N_C(z)=N^0_C(z)$.
% The normal cone of a closed convex set $C$ at $z\in C$, denoted by  $N_C(z)$, is 
% $$N_C(z):=\{\xi \in \Re^n: \inner{\xi}{u-z}\leq 0,\quad \forall u\in C\}.$$
If $\phi$ is a real-valued function which
is differentiable at $\bar z \in \Re^n$, then its affine   approximation $\ell_\phi(\cdot,\bar z)$ at $\bar z$ is given by
\begin{equation}\label{eq:defell}
\ell_\phi(z;\bar z) :=  \phi(\bar z) + \inner{\nabla \phi(\bar z)}{z-\bar z} \quad \forall  z \in \Re^n.
\end{equation}

\section{Strongly Convex Composite Optimization (SCCO)} \label{sec:SCCO}

This section presents a restarted parameter-free FISTA variant, namely RPF-SFISTA, for solving strongly convex composite optimization (SCCO) problems. 

Specifically, RPF-SFISTA assumes that problem \eqref{OptProb1} has an optimal solution $z^{*}$ and that
functions $f$, $h$, and $\phi$ satisfy the following
assumptions:
\begin{itemize}
\item[\textbf{(A1)}]
$f: \mathbb E\to \R$ is a differentiable convex function that is $\bar L$-smooth, i.e., there exists $\bar L \geq 0$ such that, for all $z,z' \in \mathbb E$,
\begin{equation}\label{ineq:uppercurvature1}
\|\nabla f(z') -  \nabla f(z)\|\le \bar L \|z'-z\|;
\end{equation}

\item[\textbf{(A2)}]  $h: \mathbb E \to \R\cup\{+\infty\}$ is a possibly nonsmooth convex function with domain denoted by $\mathcal H$;

\item[\textbf{(A3)}] $\phi$ is a $\bar \mu$-strongly convex function, where $\bar \mu>0$. 
\end{itemize}
We now describe the type of approximate solution that RPF-SFISTA aim to find. 

\vgap

Given $\phi$ satisfying the above assumptions and a tolerance parameter $\hat \epsilon>0$,
the goal of RPF-SFISTA is to find a pair $(y,v) \in \mathcal H \times \mathbb E$ such that
\begin{gather}\label{acg problem}
    \|v\|\leq \hat \epsilon, \quad v \in \nabla f(y)+\partial h(y).
\end{gather}
Any pair $(y,v)$ satisfying \eqref{acg problem} is said to be an \textbf{$\epsilon$-optimal solution} of \eqref{OptProb1}.

The rest of this section is broken up into three subsections. The first one motivates and states the RPF-SFISTA method and presents its main complexity results. The second and third subsections present the proofs of the main results.

\subsection{The RPF-SFISTA method}\label{sec:SC-FISTA}

This subsection motivates and states the RPF-SFISTA method and presents its main complexity results. 

The RPF-SFISTA is essentially a restarted version of a S-ACG variant (see for example \cite{Monteiro2016}) that also performs backtracking line-search for the smoothness parameter. 
RPF-SFISTA calls the S-ACG variant with an aggressive estimate $\mu$ for the strong convexity parameter $\bar \mu$. A novel condition is then checked at each of the variant's iterations to see if RPF-SFISTA should restart and call the variant again with a smaller estimate $\mu=\mu/2$. Each time RPF-SFISTA restarts, a new cycle of RPF-SFISTA is said to begin. If RPF-SFISTA calls the S-ACG variant with a strong convexity estimate $\mu$ such that $\mu\in (0,\bar \mu]$, then it is shown in Proposition~\ref{prop:nest_complex1} below that the current cycle of RPF-SFISTA must terminate with a pair $(y,v)$ that satisfies \eqref{acg problem} and is thus an $\epsilon$-approximate solution of \eqref{OptProb1}. Hence, RPF-SFISTA performs at most a logarithmic number of restarts/cycles. The formal description of RPF-SFISTA algorithm is now presented.

\noindent\begin{minipage}[t]{1\columnwidth}%
\rule[0.5ex]{1\columnwidth}{1pt}

\noindent \textbf{RPF-SFISTA Method}

\noindent \rule[0.5ex]{1\columnwidth}{1pt}%
\end{minipage}

\noindent \textbf{Universal Parameters}: scalars
$\chi \in (0,1)$ and $\beta>1$.

\noindent \textbf{Inputs}: let scalars $\left(\mu_0,\bar M_0\right)\in \Re_{++}^{2}$, an initial
point $z_0 \in \mathcal H$, a tolerance $\hat \epsilon>0$, and functions $(f,h)$ and $\phi:=f+h$ be given, and set $l=1$.

\noindent \textbf{Output}: a quadruple $(y,v,\xi,L)$.

\begin{itemize}
\item[{\bf 0.}]
set $j=1$, initial point $x_0=z_{l-1}$, estimates $\underbar M_{l}\in [\max\{0.25\bar M_{l-1},\bar M_0\},\bar M_{l-1}]$ and $\mu=\mu_{l-1}$, points $(\xi_0,y_0)=(x_0,x_0)$, and scalars $(A_0,\tau_0,L_0)=(0,1,\underbar M_{l})$;
\item[{\bf 1.}] set $L_{j}=L_{j-1}$;
\item[{\bf 2.}]	compute
		\begin{equation}\label{def:ak-sfista1}
		a_{j-1}=\frac{\tau_{j-1}+\sqrt{\tau_{j-1}^2+4\tau_{j-1} A_{j-1}L_{j}}}{2L_{j}}, \quad \tx_{j-1}=\frac{A_{j-1}y_{j-1}+a_{j-1} x_{j-1}}{A_{j-1}+a_{j-1}},
		\end{equation}
		\begin{equation}
		y_{j}:=\underset{u\in \mathcal H}\argmin\left\lbrace q_{j-1} (u;\tx_{j-1},L_{j}) 
		:= \ell_{f}(u;\tilde x_{j-1}) + h(u) + \frac{L_{j}}{2}\|u-\tx_{j-1}\|^2\right\rbrace;
		\label{eq:ynext-sfista1}
		\end{equation}
		if  the inequality
		\begin{equation}\label{ineq check}
		\ell_{f}(y_{j};\tilde x_{j-1})+\frac{(1-\chi) L_{j}}{4}\|y_{j}-\tilde x_{j-1}\|^2\geq f(y_{j})
		\end{equation}
		holds go to step~3; else set $L_{j} \leftarrow \beta L_{j} $ and repeat step~2;

\item[{\bf 3.}] compute
\begin{align}
{\xi_{j}}&=\begin{cases}\label{def:wj}
   y_{j} & \text{if $\phi(y_{j})\leq \phi(\xi_{j-1})$} \\
   \xi_{j-1}& \text{otherwise},
\end{cases}\\
A_{j}&=A_{j-1}+a_{j-1}, \quad \tau_{j}= \tau_{j-1} + \frac{a_{j-1}\mu}{2},  \label{eq:taunext-sfista1} \\
s_{j}&=L_{j}(\tilde x_{j-1}-y_{j}),\label{eq:sk}\\
\quad x_{j}&= \frac{1}{\tau_{j}} \left[\frac{\mu a_{j-1} y_{j}}{2} + \tau_{j-1} x_{j-1}-a_{j-1}s_{j} \right] , \label{eq:xnext-sfista1}\\
v_{j}&=\nabla f(y_{j})-\nabla f(\tilde x_{j-1})+s_{j};\label{def:uk}
\end{align}

\item[{\bf 4.}]  if the inequality
\begin{align}
&\|\xi_{j}-x_{0}\|^{2} \geq \chi A_{j}L_{j} \|y_{j}-\tilde x_{j-1}\|^2, \label{restart condition}
\end{align}
holds, then  go to step 5; otherwise \textbf{restart}, set $z_{l}=\xi_{j}$, $\bar M_{{l}}=L_{j}$, $\mu_{l}=\mu/2$, and $l \leftarrow l+1$, and go to step 0;
\item[{\bf 5.}] if the inequality
		\begin{equation}\label{u sigma criteria}
		\|v_{j}\| \leq \hat \epsilon
		\end{equation}
		holds then \textbf{stop} and output quadruple $(y,v,\xi,L):=(y_{j},v_{j},\xi_j,L_j)$; otherwise, set
		 $ j \leftarrow j+1 $ and go to step~1.

\end{itemize}
\noindent \rule[0.5ex]{1\columnwidth}{1pt}

Several remarks about RPF-SFISTA are now given. First, RPF-SFISTA is an adaptive and parameter-free method in that it requires no knowledge of the Lipschitz and strong convexity parameters and instead adaptively performs line searches for these constants.
Second, it performs two types of iterations, namely cycles indexed by $l$ and inner ACG/FISTA iterations which are indexed by $j$. The number of restarts RPF-SFISTA performs is equivalent to the number of cycles it performs minus one. Third, steps 2 and 3 of RPF-SFISTA are essentially equivalent to an iteration of a S-ACG variant that performs a line-search for the Lipschitz constant of the gradient. Fourth, step 4 of RPF-SFISTA checks a novel condition \eqref{restart condition} to determine whether it should restart or not. If relation \eqref{restart condition} fails to hold during an iteration of the $l$-th cycle of RPF-SFISTA, RPF-SFISTA restarts and begins the $(l+1)$st cycle with a smaller estimate for the strong convexity parameter $\mu$. Fifth, RPF-SFISTA performs warm-restarting, i.e., when RPF-SFISTA restarts and begins the $(l+1)$st cycle it takes as initial point for this cycle the point with the best function value that was found during the previous cycle. Finally, it will be shown that for any $j\geq 1$, the pair $(y_j,v_j)$ always satisfies the inclusion in \eqref{acg problem}. As a consequence, if RPF-SFISTA stops in step 5, output pair $(y,v)$ is an $\epsilon$-optimal solution of \eqref{OptProb1}.

Before stating the main results of the RPF-SFISTA method, the following quantities are introduced
\begin{equation}\label{D0 def}
C_{\bar \mu}(\cdot):=\frac{8}{\bar \mu}\left[\phi(\cdot)-\phi(z^{*})\right], \quad \kappa:=2\beta/(1-\chi),
\end{equation}
\begin{equation}\label{omega zeta}
\zeta_l:=\bar L+\max\{\underbar M_l,\kappa \bar L\}, \quad Q_l:=  2\sqrt{2} \sqrt{\frac{\max\{\underbar M_l,\kappa {\bar L}\}}{\mu_{l-1}}},
\end{equation}
where $\bar L$ is as in \eqref{ineq:uppercurvature1} and $z^{*}$ is an optimal solution of \eqref{OptProb1}.

The following proposition and theorem state the main complexity results of the RPF-SFISTA method and key properties of its output. The proofs of Proposition~\ref{prop:nest_complex1} and Theorem~\ref{Total Complexity} are given in Appendix~\ref{SC-FISTA Appendix Proof} and Subsection~\ref{Proof of Total Complexity}, respectively.

\begin{proposition}\label{prop:nest_complex1} 
The following statements about the $l$-th cycle of RPF-SFISTA hold:
\begin{itemize}
\item[(a)] 
it stops (in either step 4 or step 5) in at most 
\begin{equation}\label{eq:eq1}
\left\lceil\left(1+Q_l\right)\log^+_1\left(\frac{C_{\bar \mu}(z_{l-1})\zeta_l^2}{\chi \hat \epsilon^{2}}\right)+1\right\rceil+\left\lceil \frac{\log^+_0(2\bar L/((1-\chi)\underbar M_l))}{\log \beta}\right\rceil
\end{equation}
ACG iterations/resolvent evaluations 
where $\chi$ and $\beta$ are input parameters to RPF-SFISTA, $\hat \epsilon$ is the input tolerance, $\bar L$ is as in \eqref{ineq:uppercurvature1}, $C_{\bar \mu}(\cdot)$ and $\kappa$ are as in \eqref{D0 def}, and $Q_l$ and $\zeta_l$ are as in \eqref{omega zeta};
\item[(b)] if the cycle terminates in its step 5, then it outputs a quadruple $(y,v,\xi,L)$ that satisfies 
\begin{equation}\label{Func Value Decrease}
    \phi(\xi)\leq \min\left\{\phi(z_0),\phi(y)\right\}, \quad \bar M_0\leq L \leq  \max\left\{\bar M_0,\kappa \bar L\right\}
\end{equation}
and such that $(y,v)$ is an $\hat \epsilon$-optimal solution of \eqref{OptProb1}, where $z_0$ is the initial point and $\bar M_0$ is an input to RPF-SFISTA;
\item[(c)] if $\mu_{l-1} \in (0,\bar \mu]$, then the cycle always stops successfully in step 5 with a quadruple $(y,v,\xi,L)$ that satisfies \eqref{Func Value Decrease} and such that $(y,v)$ is an $\hat \epsilon$-optimal solution of \eqref{OptProb1} in at most \eqref{eq:eq1}
ACG iterations/resolvent evaluations.
\end{itemize}
\end{proposition}

The following theorem states the main complexity result of RPF-SFISTA and key properties of its output.

\begin{theorem}\label{Total Complexity}
RPF-SFISTA terminates with a quadruple $(y,v,\xi,L)$ that satisfies \eqref{Func Value Decrease} and such that $(y,v)$ is an $\hat \epsilon$-optimal solution of \eqref{OptProb1} 
in at most
\begin{equation}\label{eq:eq1-3}
\mathcal O_1\left(\left\lceil\log^{+}_1\left(2\mu_0/\bar \mu\right)\right\rceil\left[\sqrt{\frac{\max\left\{\bar M_0,\bar L\right\}}{\min\left\{\mu_0, \bar \mu/2\right\}}}\log^+_1 \left(\frac{(\bar L^2+\bar M_0^2) C_{\bar \mu}(z_0)}{\hat \epsilon^2}\right)+\log_0^{+}(\bar L/\bar M_0)\right] \right)
\end{equation}
ACG iterations/resolvent evaluations
where $\hat \epsilon$, $\mu_0$, and $\bar M_0$ are inputs to RPF-SFISTA, $z_0$ is the initial point, and $\bar \mu$, $C_{\bar \mu}(\cdot)$, and $\bar L$ are as in (B2), \eqref{D0 def}, and \eqref{ineq:uppercurvature1}, respectively.

\end{theorem}

A remark about Theorem~\ref{Total Complexity} is now given. If $\mu_0=\Omega\left(\bar \mu\right)$ and $\bar M_0=\Omega\left(\bar L\right)$, it then follows from the above result and the definition of $C_{\bar \mu}(\cdot)$ in \eqref{D0 def} that RPF-SFISTA performs at most 
\[
\mathcal O_1\left(\sqrt{\frac{\bar L}{\bar \mu}}\log^+_1 \left(\frac{\bar L^2}{\bar \mu\hat \epsilon^2}\right) \right)
\]
ACG iterations/resolvent iterations to find a pair $(y,v)$ that is an $\epsilon$-approximate optimal solution of \eqref{OptProb1}.

\subsection{Proof of Theorem~\ref{Total Complexity}}\label{Proof of Total Complexity}

This subsection is dedicated to proving Theorem~\ref{Total Complexity}. The following two lemmas present key properties of the iterates generated during the $l$-th cycle of RPF-SFISTA. The proof of the first lemma below is not given as it closely resembles the proofs of Lemmas A.3 and A.4 in \cite{SujananiMonteiro}.

\begin{lemma}\label{lem:gamma-sfista0}
% Define 
% \begin{equation}\label{omega zeta}
% \kappa=2\beta/(1-\chi), \quad \zeta:=\bar L+\max\{L_0,\kappa \bar L\}
% \end{equation}
Let $\kappa$ be as in \eqref{D0 def} and $\zeta_l$ and $Q_l$ be as in \eqref{omega zeta}. For every iteration index $j \geq 1$ generated during the $l$-th cycle of RPF-SFISTA, the following statements hold:
\begin{itemize}
\item[(a)] $\{L_j\}$ is nondecreasing;
\item[(b)] the following relations hold
\begin{align}
&\tau_{j-1} = 1+\frac{\mu A_{j-1}}{2}, \quad \frac{\tau_{j-1}  A_{j}}{ a_{j-1}^2}=L_{j},\label{tauproperty}\\[0.5em]
&\underbar M_l\leq L_{j-1}\leq \max\{\underbar M_l,\kappa \bar L\}\label{upper bound},\\[0.8em]
& v_{j} \in \nabla f(y_{j}) + \partial h(y_{j}), \quad \|v_{j}\|\leq \zeta_l \|y_{j}-\tilde x_{j-1}\|;\label{ubound-1}
\end{align}
\item[(c)] it holds that 
\begin{equation}\label{sumAkbound}
    A_jL_j \geq \max \left\{\frac{j^2}{4} ,  \left(1+Q_l^{-1}\right)^{2(j-1)}
 \right\}.
\end{equation}
\end{itemize}
\end{lemma}

\begin{lemma}\label{First Lemma}
For every iteration index $j\geq 1$ generated during the $l$-th cycle of RPF-SFISTA, it holds that
$\phi(\xi_{j})\leq \min\{\phi(\xi_{j-1}),\phi(y_{j})\}.$
As a consequence, the following relation holds
\begin{equation}\label{function value decrease XI}
\phi(\xi_{j})\leq \phi(z_{l-1}).
 \end{equation}
\end{lemma}
\begin{proof}
Let $j \geq 1$ be an iteration index generated during the $l$-th cycle of RPF-SFISTA. To see that $\phi(\xi_{j})\leq \min\{\phi(\xi_{j-1}),\phi(y_{j})\}$, consider two possible cases. First, suppose that $\phi(y_{j})\leq \phi(\xi_{j-1})$. It follows from the update rule for $\xi_{j}$ in \eqref{def:wj} that $\xi_{j}=y_{j}$ and hence that $\phi(\xi_{j})=\phi(y_{j})\leq \phi(\xi_{j-1})$. For the other case, suppose that $\phi(y_{j})>\phi(\xi_{j-1})$. Relation \eqref{def:wj} then immediately implies that $\xi_{j}=\xi_{j-1}$ and hence that $\phi(\xi_{j})=\phi(\xi_{j-1})<\phi(y_{j}).$ Combining the two cases proves the inequality holds.

It then follows from this inequality, a simple induction argument, and the fact that $\xi_0=x_0$ that $\phi(\xi_j)\leq \phi(x_0)$. This relation and the fact that $x_0$ is set as $z_{l-1}$ at the beginning of the $l$-th cycle of RPF-SFISTA immediately imply relation \eqref{function value decrease XI}.
\end{proof}

\begin{lemma}\label{Bound on L0}
For any cycle index $l \geq 1$ generated by RPF-SFISTA, the quantity $\underbar M_l$ set in step 0 satisfies
\begin{equation}\label{Bounding Upper Curvature}
\underbar M_l\leq \max\left\{\bar M_0,\kappa \bar L\right\}
\end{equation}
where $\bar M_0>0$ is an input to RPF-SFISTA and $\bar L$ and $\kappa$ are as in \eqref{ineq:uppercurvature1} and \eqref{D0 def}, respectively.
\end{lemma}
\begin{proof}
First, we show that for any cycle index $l\geq 2$ generated by RPF-SFISTA, the following relation
\begin{equation}\label{Key Lipschitz Inequality}
\bar M_0\leq \bar M_{l-1} \leq \max\left\{\underbar M_{l-1},\kappa \bar L\right\}
 \end{equation}
holds. Since $l\geq 2$ is a cycle index generated by RPF-SFISTA, this implies that its $(l-1)$-st cycle terminates in its step 4 and hence $\bar M_{l-1}$ is generated. Both relations in \eqref{Key Lipschitz Inequality} then immediately follow from this observation, the fact that the way $\bar M_{l-1}$ is chosen in step 0 implies that $\bar M_0\leq \underbar M_{l-1}$, the fact that $\bar M_{l-1}$ is set in step 4 of RPF-SFISTA as $L_j$ for some iteration index $j \geq 1$ generated during the $(l-1)$-st cycle, and the first and second inequalities in \eqref{upper bound} with $l=l-1$.

The proof of \eqref{Bounding Upper Curvature} now follows from an induction argument. The result with $l=1$ follows immediately from the fact that $\underbar M_1=\bar M_0$. Suppose now that $l\geq 2$ is a cycle index generated by RPF-SFISTA and that inequality \eqref{Bounding Upper Curvature} holds for $l-1$. It follows from the way $\underbar M_l$ is chosen in step 0 at the beginning of the $l$-th cycle of RPF-SFISTA and the first relation in \eqref{Key Lipschitz Inequality} that $\underbar M_l \leq \bar M_{l-1}$. This relation and the second relation in \eqref{Key Lipschitz Inequality} then imply that
\[\underbar M_{l} \leq \bar M_{l-1} \leq \max\left\{\underbar M_{l-1},\kappa \bar L\right\}\leq \max\left\{\bar M_0,\kappa \bar L\right\}\]
where the last inequality is due to the induction hypothesis. Hence, Lemma~\ref{Bound on L0} holds.
\end{proof}

The following lemma establishes a bound on the quantity $C_{\bar \mu}(z_l)$ in terms of the initial point $z_0$.

\begin{lemma}\label{Bounds for Restart}
    For every cycle index $l \geq 1$ generated by RPF-SFISTA, the following relations hold
\begin{align}
 &\phi(z_{l-1})\leq \phi(z_0)\label{Key Func Bound}\\
 &C_{\bar \mu}(z_{l-1})\leq C_{\bar \mu}(z_0) \label{Key Restart Bound}
\end{align}
where $z_0$ is the initial point of RPF-SFISTA and $C_{\bar \mu}(\cdot)$ is as in \eqref{D0 def}.
\end{lemma}

\begin{proof}
Both relations clearly hold for $l=1$, so let $l \geq 2$ be a cycle index generated by RPF-SFISTA. Since $l\geq 2$ is a cycle index generated by RPF-SFISTA, this implies that its $(l-1)$-st cycle terminates in step 4 and hence $z_{l-1}$ is generated. It follows from this observation, the fact that $z_{l-1}$ is set at the end of step of RPF-SFISTA as $\xi_{j}$ for some iteration index $j\geq 1$ generated during the $(l-1)$-st cycle, and relation \eqref{function value decrease XI} with $l=l-1$ that $\phi(z_{l-1})\leq \phi(z_{l-2})$. This relation and a simple induction argument then immediately imply relation \eqref{Key Func Bound}. Relation \eqref{Key Restart Bound} then follows from relation \eqref{Key Func Bound} and the definition of $C_{\bar \mu}(\cdot)$ in \eqref{D0 def}.
\end{proof}

The following proposition establishes an upper bound on the number of cycles that RPF-SFISTA performs and a bound on the number of iterations each cycle performs.

\begin{proposition}\label{Cycle Complexity}
The following statements about RPF-SFISTA hold:
\begin{itemize}
    \item[(a)] RPF-SFISTA performs at most $\lceil \log^{+}_1\left(2\mu_0/\bar \mu\right) \rceil$ cycles to find a quadruple $(y,v,\xi,L)$ that satisfies relation \eqref{Func Value Decrease} and such that $(y,v)$ is an $\hat \epsilon$-optimal solution of \eqref{OptProb1}. Moreover, for every cycle index $l\geq 1$ generated by RPF-SFISTA, it holds that $\mu_{l-1}\geq \min\left\{\mu_0, \bar \mu/2\right\}$;
    \item[(b)] each cycle of RPF-SFISTA performs at most \[\cal O_1\left(\sqrt{\frac{\max\left\{\bar M_0,\bar L\right\}}{\min\left\{\mu_0, \bar \mu/2\right\}}}\log^+_1 \left(\frac{(\bar L^2+\bar M^2_0) C_{\bar \mu}(z_0)}{\hat \epsilon^2}\right) +\log_0^{+}(\bar L/\bar M_0)\right)\] 
ACG iterations/resolvent evaluations where $\hat \epsilon$ and $\mu_0$ are inputs to RPF-SFISTA, $z_0$ is the initial point, and $\bar \mu$, $\bar L$, and $C_{\bar \mu}(\cdot)$ are as in (B2), \eqref{ineq:uppercurvature1}, and \eqref{D0 def}, respectively.
\end{itemize}
\end{proposition}

\begin{proof}
(a) It follows immediately from Proposition~\ref{prop:nest_complex1}(c) that the $l$-th cycle of RPF-SFISTA always terminates successfully in step 5 with a quadruple $(y,v,\xi,L)$ that satisfies relation \eqref{Func Value Decrease} and such that $(y,v)$ is an $\hat \epsilon$-optimal solution of \eqref{OptProb1}, if it is performed with $\mu_{l-1} \in (0, \bar \mu]$. Both conclusions of (a) then follow immediately from this observation, from the way that $\mu_{l}$ is updated when a cycle terminates in step 4 of RPF-SFISTA, and the fact that the first cycle of RPF-SFISTA is performed with $\mu=\mu_0$.

(b) Lemma~\ref{Bound on L0} implies that $\underbar M_l \leq \max\left\{\bar M_0,\kappa\bar L\right\}$ for every cycle index $l \geq 1$ generated by RPF-SFISTA. This relation and the definition of $\zeta_{l}$ in \eqref{omega zeta} imply that $\zeta_l^2 \leq \max\left\{(\bar L+M_0)^2, (\kappa+1)^2\bar L^2 \right\}$ and hence that $\zeta_l^2=\mathcal O(\bar L^2+\bar M_0^2)$. These relations, the fact that the way $\underbar M_l$ is chosen in step 0 implies that $\underbar M_l \geq \bar M_0$, Proposition~\ref{prop:nest_complex1}(a), and the definition of $Q_l$ imply that the $l$-th cycle performs at most
\[
\cal O_1\left(\sqrt{\frac{\max\left\{\bar M_0,\bar L\right\}}{\mu_{l-1}}}\log^+_1 \left(\frac{(\bar L^2+\bar M_0^2)C_{\bar \mu}(z_{l-1})}{\hat \epsilon^2}\right) +\log_0^{+}(\bar L/\bar M_0)\right)
\]
ACG iterations/resolvent evaluations. The result then follows from the above relation, relation \eqref{Key Restart Bound}, and the last conclusion of Proposition~\ref{Cycle Complexity}(a).
\end{proof}

We are now ready to prove Theorem~\ref{Proof of Total Complexity}. 
\begin{proof}[Proof of Theorem~\ref{Proof of Total Complexity}]
The first conclusion of Proposition~\ref{Cycle Complexity}(a) and Proposition~\ref{Cycle Complexity}(b) immediately imply the result.
\end{proof}

\section{Convex Composite Optimization (CCO)}\label{CCO Optimization}
This section presents an aggressive regularized method, namely A-REG, for solving convex composite optimization (CCO) problems. Specifically, A-REG considers the
problem
\begin{equation}\label{OptProbPsi}
\min\{\psi(u):= \psi_s(u) + \psi_n(u) : u \in \mathbb E\}
\end{equation}
whose solution set is nonempty and where functions $\psi$, $\psi_s$, and $\psi_n$ are assumed to satisfy the following assumptions:
\begin{itemize}
\item[\textbf{(B1)}]  $\psi_n: \mathbb E \to \R\cup\{+\infty\}$ is a possibly nonsmooth convex function with domain denoted by $\mathcal N$;

\item[\textbf{(B2)}]
$\psi_s: \mathbb E\to \R$ is a differentiable convex function that is $\bar L_{\psi_s}$-smooth, i.e., there exists $\bar L_{\psi_s} \geq 0$ such that, for all $z,z' \in \mathbb E$,
\begin{equation}\label{ineq:uppercurvaturePsi}
\|\nabla \psi_s(z') -  \nabla \psi_s(z)\|\le \bar L_{\psi_s} \|z'-z\|;
\end{equation}
\item[\textbf{(B3)}]
the sublevel sets of $\psi$ are bounded.
% there exists an optimal solution $w^{*} \in \mathcal N$ of problem \eqref{OptProbPsi}.
\end{itemize}
Assumption (B3) is a common assumption and a similar type of assumption is made in \cite{CatalystAgain, OptPar, NesterovNEWs}.
% A-REG aims to find an approximate stationary solution of \eqref{OptProbPsi} in the following sense.

Given a tolerance $\epsilon>0$, the goal of A-REG is to find a pair $(w,r) \in \mathcal N \times \mathbb E$ such that
\begin{gather}\label{psi acg problem}
    \|r\|\leq \epsilon, \quad r \in \nabla \psi_s(w)+\partial \psi_n(w).
\end{gather}
Any pair $(w,r)$ satisfying \eqref{psi acg problem} is said to be an \textbf{$\epsilon$-optimal solution} of \eqref{OptProbPsi}.

\subsection{The Aggressive Regularization (A-REG) method}\label{sec:A-REG}

This subsection motivates and states the aggressive regularization (A-REG) method and presents its main complexity results. 

Like the methods in \cite{CatalystAgain, CurvatureFree, CatalystNC, NesterovNEWs, OptPar}, A-REG is a regularization method. At each of iterations, A-REG forms strongly convex regularized subproblems which it solves using the RPF-SFISTA method developed in Subsection~\ref{sec:SC-FISTA}. A-REG calls RPF-SFISTA with an aggressive initial choice of strong convexity estimate, $\mu_0$, that is possibly much larger than the known strong convexity parameter of the objective of the regularized subproblem. RPF-SFISTA checks a key inequality at each of iterations to determine when a regularized subproblem has been approximately solved. Hence, A-REG is more aggressive than the schemes employed by \cite{CatalystAgain, CatalystNC, OptPar}, which run a standard ACG or S-ACG variant for a predetermined number of iterations to solve each subproblem. The schemes in \cite{NesterovNEWs, CurvatureFree} do not run a S-ACG variant for a predetermined number of iterations to solve each subproblem, but are less aggressive than A-REG as both schemes do not use a restarted S-ACG method for solving each of subproblem but rather they use a standard S-ACG method with an accurate estimate for strong convexity parameter to do so.

The A-REG method is now presented.

\noindent\begin{minipage}[t]{1\columnwidth}%
\rule[0.5ex]{1\columnwidth}{1pt}

\noindent \textbf{A-REG Method}

\noindent \rule[0.5ex]{1\columnwidth}{1pt}%
\end{minipage}

\noindent \textbf{Universal Parameters}: scalars
$\chi \in (0,1)$ and $\beta>1$.

\noindent \textbf{Inputs}: let scalars $B\geq 1$ and $(\delta_0,\bar N_0)\in \Re^{2}_{++}$, an initial point $\vartheta_0\in \mathcal N$, a tolerance $\epsilon>0$, and functions $(\psi_s,\psi_n)$ and $\psi:=\psi_s+\psi_n$ be given, and set $w_0=\vartheta_0$ and $k=1$.

\noindent \textbf{Output}: a pair $(w,r)$ satisfying \eqref{psi acg problem}.

\begin{itemize}
\item[{\bf 1.}]
choose
$\underbar N_k \in [\max\{0.25\bar N_{k-1},\bar N_0\}, \bar N_{k-1}]$
and call the RPF-SFISTA method described in Subsection~\ref{sec:SC-FISTA}
with inputs
\begin{equation}\label{U0 z0 inputs}
z_0=\vartheta_{k-1}, \quad
(\mu_0,\bar M_0)= (B\delta_{k-1},\underbar N_k), \quad \hat \epsilon=\epsilon/6
\end{equation}
\begin{equation}\label{functions}
(f,h)=\left(\psi_s+\frac{\delta_{k-1}}{2}\|\cdot-\vartheta_{k-1}\|^2,\psi_n\right), \quad \phi(\cdot)=\psi(\cdot)+\frac{\delta_{k-1}}{2}\|\cdot-\vartheta_{k-1}\|^2
\end{equation}
and let $(w_k,u_k,\vartheta_k, \bar N_k)$ denote its output $(y,v, \xi, L)$;

\item[{\bf 2.}]
set 
\begin{equation}\label{def of rk}
r_{k}:=u_{k}+\delta_{k-1}(\vartheta_{k-1}-w_k);
\end{equation}
\item[{\bf 3.}]	
%compute
% \begin{align}
% p_{k}:&=p_{k-1}+c_{k}(Az_k-b)\label{eq:dual_update2};
% \end{align}
if $\|r_k\|\leq \epsilon$
then stop and output $(w,r)=(w_k,r_k)$; else
set $\delta_{k}=\delta_{k-1}/2$, $k\gets k+1$, and go to step~1. 
\end{itemize}
\noindent \rule[0.5ex]{1\columnwidth}{1pt}

Several remarks about A-REG are now given. First, A-REG is a parameter-free and an aggressive method in that it requires no knowledge of the Lipschitz constant and employs a restarted parameter-free method, RPF-SFISTA, to solve its strongly convex subproblems.
Second, the $k$-th iteration of A-REG calls RPF-SFISTA with a function $\phi$ as in \eqref{functions} that is $\delta_{k-1}$-strongly convex. However, A-REG calls RPF-SFISTA with an aggressive initial strong convexity estimate $\mu_0=B\delta_{k-1}$ where $B \geq 1$. Hence, A-REG differs from the methods in \cite{NesterovNEWs,CatalystAgain, CurvatureFree, CatalystNC, OptPar} as these methods \textbf{do not call restarted} S-ACG methods with an aggressive initial choice for the strong convexity estimate to solve their subproblems, but rather they call standard S-ACG or even ACG methods to solve them. Finally, it will be shown in the next subsection that for any $k \geq 1$, the pair $(w_k,r_k)$ always satisfies the inclusion $r_k\in \nabla \psi_s(w_k)+\partial \psi_n(w_k)$. Thus, if A-REG stops in its step 3, it follows that output pair $(w,r)=(w_k,r_k)$ is an $\epsilon$-optimal solution of \eqref{OptProbPsi}.

\begin{remar}\label{Assumption}
Assumption (B3) implies that the sublevel set $S:=\left\{x \in \mathcal N: \psi(x)\leq \psi(\vartheta_0)\right\}$ is bounded, i.e., any $\vartheta \in S$ satisfies $\|\vartheta\|\leq D$ where $D>0$.
% \begin{align}
  
% \end{align}
\end{remar}
Before stating the main result of the A-REG method, the following quantities are introduced
\begin{equation}\label{Important quantities}
\bar L_{\delta_0}=\kappa\left(\bar L_{\psi_s}+\delta_{0}\right), \quad \psi^{0}=\psi(\vartheta_0)-\psi(w^{*}), \quad \bar{\mathcal L^{2}}=\max\left\{\bar L^2_{\delta_0}+\bar N_0^2,2\bar L_{\delta_0}^2\right\}
\end{equation}
where $\kappa$ is as in \eqref{D0 def}, $\delta_0$ is an input parameter of A-REG, $\vartheta_0$ is the initial point, and $w^{*}$ is an optimal solution of \eqref{OptProbPsi}.

The following theorem states the main complexity result of A-REG.

\begin{theorem}\label{Complexity A-REG}
A-REG terminates with a pair $(w,r)$ that is an $ \epsilon$-optimal solution of \eqref{OptProbPsi} 
in at most
% \begin{align*}\label{Theorem A-REG}
% \mathcal O_1\Biggl(\lceil \log^{+}_1\left(8D\delta_0/\epsilon\right) \rceil\Biggl[\left\lceil\log^{+}_12B\right\rceil\sqrt{\frac{2\max\left\{\bar N_0,\bar L_{\delta_0}\right\}}{\min\left\{\delta_0, \epsilon/(8D)\right\}}}&\log^+_1 \left(\frac{288\max\left\{\bar L^2_{\delta_0}+\bar N_0^2,2\bar L_{\delta_0}^2\right\} \psi^{0}}{\min\left\{\delta_0, \epsilon/(8D)\right\}\epsilon^2}\right) \\
% &+\log(\bar L_{\delta_0}/\bar N_0)\Biggr]\Biggr)
% \end{align*}
 \begin{align*}\label{Theorem A-REG}
\mathcal O_1\left(\lceil \log^{+}_1\left(8D\delta_0/\epsilon\right) \rceil\left\lceil\log^{+}_12B\right\rceil\left[\sqrt{\frac{2\max\left\{\bar N_0,\bar L_{\delta_0}\right\}}{\min\left\{\delta_0, \epsilon/(8D)\right\}}}\log^+_1 \left(\frac{288\psi^{0}\bar{\mathcal L^2}}{\min\left\{\delta_0, \epsilon/(8D)\right\}\epsilon^2}\right)+\log(\bar L_{\delta_0}/\bar N_0)\right]\right)
\end{align*}

ACG iterations/resolvent evaluations
where $\epsilon$, $\delta_0$, and $\bar N_0$ are inputs to A-REG, $\bar L_{\delta_0}$, $\psi^{0}$, and $\bar{\mathcal L^2}$ are as in \eqref{Important quantities}, and $B\geq 1$ and $D>0$ are scalars as in \eqref{U0 z0 inputs} and Assumption~\ref{Assumption}, respectively.

\end{theorem}

Several remarks about Theorem~\ref{Complexity A-REG} are now given. It follows from the definition of $\bar L_{\delta_0}$ in \eqref{Important quantities} that, up to logarithmic terms, A-REG performs at most 
\[
\mathcal O_1\left(\sqrt{\frac{\bar L_{\psi_s}}{\epsilon}} \right)
\]
ACG iterations/resolvent iterations to find a pair $(w,r)$ that is an $\epsilon$-approximate optimal solution of \eqref{OptProbPsi}. Ignoring logarithmic terms, this complexity is optimal for finding an approximate optimal solution according to the criterion in \eqref{psi acg problem}.
 
\subsection{Proof of Theorem~\ref{Complexity A-REG}}
This subsection is dedicated to proving Theorem~\ref{Complexity A-REG}. Before stating the next lemma, we introduce the following quantities which are used throughout this subsection
\begin{equation}\label{Psi Delta}
\psi_{\delta}(w;\vartheta):=\psi(w)+\frac{\delta}{2}\|w-\vartheta\|^2, \quad \mathcal C_{\delta}(\cdot):=\frac{8}{\delta}\left[\psi(\cdot)-\psi(w^{*})\right]
\end{equation}
where $\delta>0$ is a scalar and $w^{*}$ is an optimal solution of \eqref{OptProbPsi}. 

Since A-REG calls the RPF-SFISTA method during each of its iterations, the following lemma specializes Theorem~\ref{Total Complexity}, which states the complexity of RPF-SFISTA and key properties of its output, to this set-up.

\begin{lemma}\label{translation RPF-SFISTA}
    The following statements about the $k$-th iteration of A-REG hold
    \begin{itemize}
        \item[(a)] the function $f$ in \eqref{functions} is $(\bar L_{\psi_s}+\delta_{0})$-smooth and $\delta_{k-1}$-strongly convex. As a consequence, the function, $\phi$ in \eqref{functions} is $\delta_{k-1}$-strongly convex;
        \item[(b)] the call made to the RPF-SFISTA method in step 1 outputs a quadruple $(w_k,u_k,\vartheta_k, \bar N_k)$ that satisfies the following relations
        \begin{equation}\label{VarTheta Psi}
          \psi(\vartheta_{k})\leq \psi(\vartheta_{k-1}), \quad \psi_{\delta_{k-1}}(\vartheta_{k};\vartheta_{k-1})\leq \psi_{\delta_{k-1}}(w_{k};\vartheta_{k-1}), \quad \bar N_0\leq \bar N_{k} \leq \max\{\underbar N_k,\bar L_{\delta_0}\},
        \end{equation}
        \begin{equation}\label{u inclusion}
          u_{k}\in \nabla \psi_s(w_k)+\delta_{k-1}(w_k-\vartheta_{k-1})+\partial \psi_n(w_k), \quad \|u_k\|\leq \frac{\epsilon}{6}
        \end{equation}
where $\bar L_{\delta_0}$ and $\psi_{\delta}(w;\vartheta)$ are as in \eqref{Important quantities} and \eqref{Psi Delta}, respectively;
        \item[(c)] the call made to the RPF-SFISTA method performs at most 
\begin{align*}
\mathcal O_1\left(\left\lceil\log^{+}_1\left(2B\right)\right\rceil\left[\sqrt{\frac{2\max\left\{\underbar N_k,\bar L_{\delta_0}\right\}}{\delta_{k-1}}}\log^+_1 \left(\frac{36\left(\bar L_{\delta_0}^2+\underbar N^{2}_k\right) \mathcal C_{\delta_{k-1}}(\vartheta_{k-1})}{\epsilon^2}\right)+\log(\bar L_{\delta_0}/\bar N_0)\right]\right)
\end{align*}
ACG iterations/resolvent evaluations to find a quadruple $(w_k,u_k,\vartheta_k, \bar N_k)$ satisfying relations \eqref{VarTheta Psi} and \eqref{u inclusion} , where $B\geq 1$ and $\mathcal C_{\delta}(\cdot)$ are as in \eqref{U0 z0 inputs} and \eqref{Psi Delta}, respectively.
\end{itemize}

\end{lemma}
\begin{proof}
(a) It follows immediately from the facts that $\psi_{s}$ is $\bar L_{\psi_s}$-smooth and $\psi_{s}$ and $\psi$ are convex functions and the fact that during the $k$-th iteration of A-REG, the call to RPF-SFISTA is made with $f(\cdot):=\psi_s(\cdot)+0.5\delta_{k-1}\|\cdot-\vartheta_{k-1}\|^2$ and $\phi(\cdot):=\psi(\cdot)+0.5\delta_{k-1}\|\cdot-\vartheta_{k-1}\|^2$ that $f$ is $(\bar L_{\psi_s}+\delta_{k-1})$-smooth and $\delta_{k-1}$-strongly convex and $\phi$ is $\delta_{k-1}$-strongly convex. The result then follows immediately from this conclusion and the fact that the way $\delta_{k}$ is updated in step 3 of A-REG implies that $\delta_{k}\leq \delta_{0}$ for any iteration index $k\geq 1$.

(b) It follows from the definitions of $\phi$ and $\psi_{\delta}$ in \eqref{functions} and \eqref{Psi Delta}, respectively, that the call to RPF-SFISTA during the $k$-th iteration of A-REG is made with $\phi(\cdot):=\psi_{\delta_{k-1}}(\cdot;\vartheta_{k-1})$. It then follows from this observation, the fact that RPF-SFISTA is called with functions $(f,h)$ and $\phi$ as in \eqref{functions} and inputs $(z_0,\bar M_0)=(\vartheta_{k-1},\bar N_k)$, Theorem~\ref{Total Complexity}, and part (a) that RPF-SFISTA outputs a quadruple $(w_k,u_k,\vartheta_k,\bar N_k)=(y,v,\xi,L)$ that satisfies
\[
\psi_{\delta_{k-1}}(\vartheta_k;\vartheta_{k-1})\leq \min\left\{\psi_{\delta_{k-1}}(\vartheta_{k-1};\vartheta_{k-1}),\psi_{\delta_{k-1}}(w_k;\vartheta_{k-1})\right\}, \quad \underbar N_k\leq \bar N_k \leq  \max\left\{\underbar N_k,\kappa (\bar L_{\psi_s}+\delta_0)\right\}.
\]
The relations in \eqref{VarTheta Psi} then follow from the above relations, the fact that the way $\underbar N_k$ is chosen in step 1 implies that $\underbar N_k\geq \bar N_0$, the definition of $\bar L_{\delta_0}$ in \eqref{Important quantities}, and the fact that the definition of $\psi_{\delta}$ in \eqref{Psi Delta} implies that $\psi(\vartheta_{k})\leq \psi_{\delta_{k-1}}(\vartheta_k;\vartheta_{k-1})$ and $\psi_{\delta_{k-1}}(\vartheta_{k-1};\vartheta_{k-1})=\psi(\vartheta_{k-1})$.

It also follows immediately from Theorem~\ref{Total Complexity}, the definition of $\epsilon$-optimal solution, and the fact that RPF-SFISTA is called with  tolerance $\hat \epsilon=\epsilon/6$ and function pair $(f,h)$ as in \eqref{functions} that pair $(w_k,u_k)=(y,v)$ satisfies both relations in \eqref{u inclusion}.

(c) Consider the call made to RPF-SFISTA during the $k$-th iteration of A-REG. It follows directly from part (a), the fact that $\kappa \geq 1$, and the definition of $\bar L_{\delta_0}$ in \eqref{Important quantities} that the function $f$ in \eqref{functions} is $\bar L_{\delta_0}$-smooth and hence satisfies relation \eqref{ineq:uppercurvature1} with $\bar L=\bar L_{\delta_0}$. Part (a) implies that $\phi$ is $\delta_{k-1}$-strongly convex and thus satisfies assumption (A3) with $\bar \mu=\delta_{k-1}$. It also follows from the definitions of $\phi$ and $\psi_{\delta}$ in \eqref{functions} and \eqref{Psi Delta}, respectively that $\phi(\vartheta_{k-1})=\psi(\vartheta_{k-1})$ and that $\psi(w^{*})\leq \min_{x}\psi_{\delta_{k-1}}(x;\vartheta_{k-1})$ where $w^{*}$ is an optimal solution of \eqref{OptProbPsi}. It then follows from these conclusions, the fact that the call to RPF-SFISTA is made with tolerance $\hat \epsilon=\epsilon/6$, functions $(f,h)$ and $\phi$ as in \eqref{functions}, initial point $z_0=\vartheta_{k-1}$, and inputs $(\mu_0,\bar M_0)=(B\delta_{k-1},\underbar N_k)$, and the definitions of $C_{\bar \mu}(\cdot)$ and $\mathcal C_{\delta}(\cdot)$ in \eqref{D0 def} and \eqref{Psi Delta}, respectively that RPF-SFISTA performs at most
\[
\mathcal O_1\left(\left\lceil\log^{+}_1\left(2B\right)\right\rceil\left[\sqrt{\frac{2\max\left\{\underbar N_k,\bar L_{\delta_0}\right\}}{\min\left\{B\delta_{k-1}, \delta_{k-1}/2\right\}}}\log^+_1 \left(\frac{36\left(\bar L_{\delta_0}^2+\underbar N^{2}_k\right) \mathcal C_{\delta_{k-1}}(\vartheta_{k-1})}{\epsilon^2}\right)+\log(\bar L_{\delta_0}/ \underbar N_k)\right]\right)
\]
ACG iterations/resolvent evaluations. The result then follows from this conclusion, the fact that the way $\underbar N_k$ is chosen in step 1 of A-REG implies that $\underbar N_k \geq \bar N_0$, the fact that $B\geq 1$, and part (b).
%the number of iterations in 41...
% Theorem~\ref{}, the fact that RPF-SFISTA is called with tolerance $\hat \epsilon=\epsilon/6$ and functions $(f,h)$ and $\phi$ as in \eqref{}, and the definition of $\epsilon$-optimal solution imply
% 
% that RPF-SFISTA terminates with a quadruple $(w_k,u_k,\vartheta_k,\bar N_k=(y,v,\xi,L)$ such that 
% \[\]
% 
% Both relations in \eqref{} follow immediately from the fact that RPF-SFISTA is called with functions $(f,h)$ and $\phi$ as in \eqref{}
% 
% (c) First observe that, something about phi...
% 
% It then follows from this observation, the complexity result in \eqref{}, the fact that the call to RPF-FISA is made with functions $(f,h)$ and $\phi$ as in \eqref{} and inputs $z_0=\vartheta_{k-1}$ and $(\mu_0,\bar M_0)= (B\delta_{k-1},\underbar N_k)$, and the definitions of $C_{\bar \mu}(\cdot)$ and $\mathcal C_{\delta}(\cdot)$ in \eqref{} and \eqref{}, respectively that RPF-SFISTA performs at most
\end{proof}

\begin{lemma}\label{Bound on NK}
For every iteration index $k \geq 1$ generated by A-REG, the quantity $\underbar N_k$ satisfies
\begin{equation}\label{Bound NK}
\underbar N_{k}\leq \max\left\{\bar N_0,\bar L_{\delta_{0}}\right\}
\end{equation}
where $\bar N_0$ is an input to A-REG and $\bar L_{\delta_0}$ is as in \eqref{Important quantities}.
\end{lemma}
\begin{proof}
The proof follows from an induction argument. The fact that $\underbar N_1=\bar N_0$ immediately implies that relation \eqref{Bound NK} holds for $k=1$. Suppose now that $k\geq 2$ is an iteration index generated by A-REG and that inequality \eqref{Bound NK} holds for $k-1$. It follows from the way $\underbar N_k$ is chosen in step 1 of the $k$-th iteration of A-REG and the first inequality in the last relation in \eqref{VarTheta Psi} that $\underbar N_k \leq \bar N_{k-1}$. This relation and the last inequality in \eqref{VarTheta Psi} with $k=k-1$ then imply that
\[\underbar N_{k} \leq \bar N_{k-1} \overset{\eqref{VarTheta Psi}}{\leq} \max\{\underbar N_{k-1},\bar L_{\delta_0}\} \leq \max\left\{\bar N_0,\bar L_{\delta_0}\right\}\]
where the last inequality is due to the induction hypothesis. Hence, relation \eqref{Bound NK} holds for every iteration index $k\geq 1$ generated by A-REG.
\end{proof}
A useful fact that is used in the proof of following result is that if a function $\Psi:\mathbb E \to\R\cup\{+\infty\}$ is $\nu$-convex with modulus $\nu>0$, then it has an unique global minimum $x^*$ and
\begin{equation}\label{ineq:nu-convex}
\Psi(x^*) +\frac{\nu}{2}\|\cdot - x^*\|^2\leq \Psi(\cdot). 
\end{equation}
\begin{lemma}\label{Bound on SubLevel Set}
For any iteration index $k \geq 1$ generated by A-REG, the following relations hold
\begin{align}
    &r_k \in \nabla \psi_s(w_k)+\partial \psi_n(w_k) \label{inclusion r}\\
    &\|\vartheta_k\|\leq D, \quad \|w_k\|\leq D+\frac{\epsilon}{3\delta_{k-1}}\label{sublevel bound}
\end{align}
where $D>0$ is as in Assumption~\ref{Assumption} and $ \epsilon>0$ is the input tolerance to A-REG.
\end{lemma}

\begin{proof}
Relation \eqref{inclusion r} follows immediately from the inclusion in \eqref{u inclusion} and the definition of $r_k$ in \eqref{def of rk}. It follows immediately from the first inequality in \eqref{VarTheta Psi} and a simple induction argument that $\psi(\vartheta_k)\leq \psi(\vartheta_0)$ for any iteration index $k\geq 1$ generated by A-REG. This relation and Assumption~\ref{Assumption} then immediately imply that the first relation in \eqref{sublevel bound} holds.

The second relation in \eqref{sublevel bound} clearly holds if $\|w_k-\vartheta_k\|=0$ since this relation together with the first relation in \eqref{sublevel bound} implies that $\|w_k\|=\|\vartheta_k\|\leq D$.
Hence, assume that $\|w_k-\vartheta_k\|> 0$.
It follows from the definition of $\psi_{\delta}(w;\vartheta)$ in \eqref{Psi Delta} and the facts that $\psi_s$ and $\psi_n$ are convex that the inclusion in \eqref{u inclusion} is equivalent to $u_{k} \in \partial \psi_{\delta_{k-1}}(w_k;\vartheta_{k-1})$. This inclusion implies that \[w_k=\argmin_{x}\psi_{\delta_{k-1}}(x;\vartheta_{k-1})-\inner{u_k}{x-w_k}.\] It then follows from this relation, the fact that the definition of $\psi_{\delta}(w;\vartheta)$ in \eqref{Psi Delta} implies that $\psi_{\delta_{k-1}}(x;\vartheta_{k-1})-\inner{u_k}{x-w_k}$ is $\delta_{k-1}$-strongly convex, and relation \eqref{ineq:nu-convex} with $\Psi(\cdot)=\psi_{\delta_{k-1}}(\cdot;\vartheta_{k-1})-\inner{u_k}{\cdot-w_k}$, $x^{*}=w_k$, and $\nu=\delta_{k-1}$ that
\begin{equation}\label{strong conv w}
\psi_{\delta_{k-1}}(w_k;\vartheta_{k-1})+\frac{\delta_{k-1}}{2}\|x-w_k\|^2\leq \psi_{\delta_{k-1}}(x;\vartheta_{k-1})-\inner{u_k}{x-w_k}
\end{equation}
for all $x \in \mathbb E.$
Relation \eqref{strong conv w} with $x=\vartheta_{k}$ and the second relation in \eqref{VarTheta Psi} imply that
$\inner{u_k}{w_k-\vartheta_k}\geq 0.5\delta_{k-1}\|\vartheta_k-w_k\|^2$
which together with the fact that $\|w_k-\vartheta_k\|>0$ and Cauchy-Schwarz inequality implies that $0.5\delta_{k-1}\|\vartheta_k-w_k\|\leq \|u_k\|$. This relation, the first relation in \eqref{sublevel bound}, the inequality in \eqref{u inclusion}, and reverse triangle inequality then imply that
\[
    \|w_k\|\leq \|\vartheta_k\|+\|\vartheta_k-w_k\|\overset{\eqref{sublevel bound}}{\leq} D+\frac{2}{\delta_{k-1}}\|u_k\|\overset{\eqref{u inclusion}}{\leq} D+\frac{\epsilon}{3\delta_{k-1}}
\]
from which the second relation in \eqref{sublevel bound} immediately follows.
\end{proof}

The following proposition establishes an upper bound on the number of iterations that A-REG performs and, consequently, a lower bound on the quantity $\delta_k$.

\begin{proposition}\label{PT Outer Iterations}
A-REG performs at most $\lceil \log^{+}_1\left(8D\delta_0/\epsilon\right) \rceil$ iterations to find a pair $(w,r)$ that satisfies \eqref{psi acg problem}. Moreover, for every iteration index $k\geq 1$ generated by A-REG, it holds that $\delta_{k-1}\geq \min\left\{\delta_0, \epsilon/(8D)\right\}$.
\end{proposition}

\begin{proof}
In view of the fact that the first iteration of A-REG is performed with $\delta_0$, the way $\delta_{k}$ is updated at the end of step 3 of A-REG, and the fact that for any iteration index $k$ generated by A-REG pair $(w_k,r_k)$ always satisfies inclusion \eqref{inclusion r}, to show both conclusions of the proposition it suffices to show that if the $k$-th iteration of A-REG is performed with $\delta_{k-1}\leq \epsilon/(4D)$, then $\|r_k\|\leq \epsilon$ and hence A-REG terminates in its $k$-th iteration. 

Hence, assume that the $k$-th iteration of A-REG is performed with $\delta_{k-1}\leq \epsilon/(4D)$. It then follows from this relation, the definition of $r_k$ in \eqref{def of rk}, triangle inequality, the inequality in \eqref{u inclusion}, and both relations in \eqref{sublevel bound} that
\begin{align*}
    \|r_k\|&\overset{\eqref{def of rk}}{\leq} \|u_k\|+\delta_{k-1}\|\vartheta_{k-1}-w_k\|\overset{\eqref{u inclusion}}{\leq} \frac{\epsilon}{6}+\delta_{k-1}\|\vartheta_{k-1}-w_k\|\\
    &\leq \frac{\epsilon}{6}+\delta_{k-1}\left(\|\vartheta_{k-1}\|+\|w_k\|\right)\overset{\eqref{sublevel bound}}{\leq} \frac{\epsilon}{6}+\delta_{k-1}\left(2D+\frac{\epsilon}{3\delta_{k-1}}\right)=\frac{\epsilon}{2}+2D\delta_{k-1}\leq \epsilon.
\end{align*}
It then follows from the above relation that A-REG terminates in step 3 of its $k$-th iteration with a pair $(w,r)=(w_k,r_k)$ satisfying \eqref{psi acg problem} and hence both conclusions of the proposition hold.
\end{proof}
We are now ready to prove Theorem~\ref{Complexity A-REG}.
\begin{proof}
It follows from the first relation in \eqref{VarTheta Psi} and an induction argument that $\psi(\vartheta_k)\leq \psi(\vartheta_0)$ for any iteration index $k \geq 1$ generated by A-REG. It then follows from this fact and the definitions of $\psi^{0}$ and $\mathcal C_{\delta}(\cdot)$ in \eqref{Important quantities} and \eqref{Psi Delta}, respectively, that $36 \mathcal C_{\delta_{k-1}}(\vartheta_{k-1})\leq (288 \psi^{0})/\delta_{k-1}$. It then follows from this relation, the definition of $\bar{\mathcal L^2}$ in \eqref{Important quantities}, the bounds on $\underbar N_k$ and $\delta_k$ in \eqref{Bound NK} and Proposition~\ref{PT Outer Iterations}, respectively, and Lemma~\ref{translation RPF-SFISTA}(c) that the call made to the RPF-SFISTA method during the $k$-th iteration of A-REG performs at most
\[
\left(\lceil \log^{+}_1\left(8D\delta_0/\epsilon\right) \rceil\left\lceil\log^{+}_12B\right\rceil\left[\sqrt{\frac{2\max\left\{\bar N_0,\bar L_{\delta_0}\right\}}{\min\left\{\delta_0, \epsilon/(8D)\right\}}}\log^+_1 \left(\frac{288\psi^{0}\bar{\mathcal L^2}}{\min\left\{\delta_0, \epsilon/(8D)\right\}\epsilon^2}\right)+\log(\bar L_{\delta_0}/\bar N_0)\right] \right)
\]
ACG iterations/resolvent evaluations. The result then immediately follows from this observation and the first conclusion of Proposition~\ref{PT Outer Iterations}.
\end{proof}

\section{Numerical Experiments} \label{sec:numerical}
This section benchmarks the numerical performance of RPF-SFISTA against four other state-of-the-art methods for solving four classes of strongly convex/convex composite optimization problems. It contains five subsections. The first subsection reports the numerical performance of all 5 methods on sparse logistic regression problems, the second one reports the performance of the methods on Lasso problems, while the third and fourth subsections report the numerical performance of all the methods on dense vector quadratic programs constrained to the simplex and box, respectively. The last subsection contains comments about the numerical results.

We now describe our implementation of RPF-SFISTA. First, the input parameters of RPF-SFISTA are chosen as
\[
\beta=1.25, \quad \chi=0.001, \quad \underbar M_1=10.
\]
Second, for each problem instance, the initial strong convexity estimate $\mu_0$ is always taken as
\[\mu_0=\frac{4\left[f(y_1)-\ell_{f}(y_{1};\tilde x_{0})\right]}{(1-\chi)\|y_{1}-\tilde x_{0}\|^2}\]
where $y_1$ and $\tilde x_0$ are generated at the end of step 3 of the first cycle of RPF-SFISTA. Third, each time RPF-SFISTA restarts, the next strong convexity estimate $\mu_l$ is taken to be $\mu_l=0.1\mu_{l-1}$.
Finally, for $l \geq 2$, $\underbar M_l$ is chosen as $0.4\bar M_{l-1}$.

RPF-SFISTA is bench-marked against the following four state-of-the-art algorithms. Specifically, we consider the FISTA method with backtracking presented in \cite{beck2009fast} (nicknamed FISTA-BT), the restarted method of \cite{Candes} (nicknamed FISTA-R), the RADA-FISTA method of \cite{GreedyFISTA} (nicknamed RA-FISTA), and the Greedy-FISTA method of \cite{GreedyFISTA} (nicknamed GR-FISTA). FISTA-R restarts based on a heuristic function value restarted scheme. RADA-FISTA and Greedy-FISTA are restarted FISTA methods that restart based on the heuristic gradient restarted scheme presented in \cite{Candes}.

Next, we describe the implementation details of the four algorithms which we compare RPF-SFISTA with. 
The implementation of FISTA-BT adaptively searches for a Lipschitz estimate $L_j$ by a procedure similar to step 2 of RPF-SFISTA. Specifically, it checks 
an inequality similar to \eqref{ineq check} with $\chi=0.001$ and if the inequality does not hold, it doubles its Lipschitz estimate $L_j$, re-generates its potential next iterate $y_j$, and checks the inequality again. Also, like RPF-SFISTA, FISTA-BT sets $10$ as its initial guess for the Lipschitz constant. FISTA-R
restarts FISTA-BT (as described above) if FISTA-BT finds a point $y_j$ with worse objective value than the previous point, i.e., if  $\phi(y_j)>\phi(y_{j-1})$. It also sets $10$ as an initial estimate for the Lipchitz constant. The implementations of RADA-FISTA and Greedy-FISTA are taken directly from the authors' Github repository, https://github.com/jliang993/Faster-FISTA. The input parameters of RADA-FISTA are chosen as
\[p=0.5, \quad q=0.5, \quad r=4, \quad \gamma=\frac{1}{\bar L}\]
where $\bar L$ is the global Lipschitz constant of the gradient of $f$. All input parameters are chosen so as to meet the required ranges presented in Algorithm 4.2 of \cite{GreedyFISTA}.
Greedy-FISTA only takes in a single input, a stepsize $\gamma$, which is chosen as $1.3/\bar L$. This more aggressive choice of stepsize is in the range suggested in the paragraph following Algorithm 4.3 in \cite{GreedyFISTA} and is the stepsize used in several of the experiments presented in the authors' Github repository. The code further has a safeguard which allows the stepsize to be decreased if a certain condition is satisfied.

We now describe the type of solution each of the methods aims to find. That is, given functions $f$ and $h$ satisfying assumptions (B1)-(B3) described in Section~\ref{sec:SCCO}, an initial point $z_0 \in \mathcal H$, and tolerance $\hat \epsilon>0$, each of the methods aims to find a pair $(z,v)$ satisfying:
\begin{equation}\label{solution type comp study}
v\in\nabla f(z)+\partial h(z), \quad \frac{\|v\|}{1+\|\nabla f(z_{0})\|}\leq \hat\epsilon
\end{equation}
where $\|\cdot\|$ signifies the Euclidean norm.

The tables below report the runtimes and the total number of ACG iterations/resolvent evaluations needed to find a pair $(z,v)$ satisfying \eqref{solution type comp study}. A tolerance of either $10^{-8}$ or $10^{-13}$ is set and a time limit of 7200 seconds (2 hours) is given. An entry of a table marked with $*/N$ means that the corresponding method finds an approximate solution with relative accuracy 
strictly larger than the desired accuracy in which case $N$
expresses the minimum relative accuracy that the method achieved within the time limit of 2 hours. The bold numbers in the tables of this section indicate the algorithm that performed the best for that particular metric (i.e. runtime or ACG iterations).

It will be seen from the numerical results presented in Subsections~\ref{logistic regression}, \ref{Lasso}, \ref{subsec:nonconvex_qpsimplex}, and \ref{subsec:strconvex_qpbox} that our method, RPF-SFISTA, was the fastest method and the only method able to find a solution of the desired accuracy (either $10^{-8}$ or $10^{-13}$) on every instance considered. To compare RPF-SFISTA and the second-best performing method on a particular problem class more closely,
we also report in each table caption the following average time ratio (ATR) between
the two methods defined as
\begin{equation}\label{speed comparison}
ATR=\frac{1}{N}\sum_{i=1}^{N}b_i/r_i,   
\end{equation}
where $N$ is the number of class instances that both methods were tested on and $b_i$ and $r_i$ are the runtimes of the second-best performing method and RPF-SFISTA for instance $i$, respectively. If a method did not finish within the time limit of 7200 seconds on a particular instance, then its runtime for that instance is conservatively recorded as 7200 seconds in the computation of the ATR metric.

All experiments were performed in MATLAB 2024a and run on a Macbook Pro with a Apple M3 Max chip and 128 GB of memory. 

\subsection{Sparse logistic regression}\label{logistic regression}

Consider the problem
\begin{align*}
\min_{z\in \mathbb R^{n}}\  & \left[f(z):=\sum_{i=1}^m \log \left(1+ \exp(-b_i\inner{a_i}{z} \right)\right]\\
\text{s.t.}\  & \|z\|_{1}\leq C
\end{align*}
where $a_i \in \mathbb R^{n}$ and $b_i \in\{-1,1\}$ for $i\in [m]$, and $C>0$ is a regularization parameter. For our experiments in \S~\ref{logistic regression}, we vary the dimension pair $(m,n)$ and $C$ is taken to be either $0$, $0.5$, or $1$. For each instance, the initial point $x_0$ is chosen to be a random point satisfying $\|x_0\|_{1}\leq C$. The Lipschitz constant $\bar L$ of $f$ is upper bounded by $0.25\lambda_{\max}(D^{T}D)$ where $D\in \mathbb R^{m\times n}$ satisfies $D_{ij}=-a_i^{j}b_i$ and $a_i^{j}$ denotes the $j$-th entry of $a_i$.

\begin{table}[!tbh]
\captionsetup{font=scriptsize}
\begin{centering}
\begin{tabular}{>{\centering}p{2.8cm}>{\centering}p{1.8cm}|>{\centering}p{1.8cm}>{\centering}p{1.8cm}>{\centering}p{1.8cm}>{\centering}p{1.8cm}>{\centering}p{1.8cm}}
\multicolumn{1}{c}{\textbf{\scriptsize{}Parameters}} & \multicolumn{1}{c|}{\textbf{\scriptsize{}Lipschitz}} & \multicolumn{5}{c}{\textbf{\scriptsize{}Iteration Count/Runtime (seconds)}}\tabularnewline
\hline 
{\scriptsize{}($m$,$n$,$C$)} & {\scriptsize{}$\bar L$} & {\scriptsize{}RPF-SFISTA} & {\scriptsize{}FISTA-BT} & {\scriptsize FISTA-R}& {\scriptsize{}RA-FISTA} & {\scriptsize{}GR-FISTA} \tabularnewline
\hline 
\hline 
{\scriptsize{(500, 50,000, 0.5)}} & {\scriptsize{$1.56*10^6$}} & {\scriptsize{\textbf{429/32.75}}} & {\scriptsize{12140/1013.03}} & {\scriptsize{5357/529.51}} & {\scriptsize{6483/463.16}} & {\scriptsize{6303/410.17}}\tabularnewline

{\scriptsize{(500, 50,000, 1)}} & {\scriptsize{$1.56*10^6$}} & \textbf{\scriptsize{802/62.23}} & {\scriptsize{85066/6958.55}} & {\scriptsize{21177/2065.41}} & {\scriptsize{7069/490.30}} & {\scriptsize{6474/410.03}} \tabularnewline

{\scriptsize{(500, 50,000, 2)}} & {\scriptsize{$1.56*10^6$}} & \textbf{\scriptsize{909/71.02}} & {\scriptsize{59479/4766.41}} & {\scriptsize{19469/1886.22}} & {\scriptsize{8445/589.76}} & {\scriptsize{7889/508.38}}\tabularnewline

\hline
{\scriptsize{(1000, 250,000, 0.5)}} & {\scriptsize{$1.56*10^7$}} & \textbf{\scriptsize{1110/739.98}} & {\scriptsize{*/6.51e-07}} & {\scriptsize{*/3.27e-08}} & {\scriptsize{*/4.91e-08}} & {\scriptsize{14898/6694.58}}\tabularnewline

{\scriptsize{(1000, 250,000, 1)}} & {\scriptsize{$1.56*10^7$}} & \textbf{\scriptsize{1375/943.30}} & {\scriptsize{*/2.67e-06}} & {\scriptsize{*/8.08e-07}} & {\scriptsize{*/2.30e-07}} & {\scriptsize{*/1.52e-07}}\tabularnewline

{\scriptsize{(1000, 250,000, 2)}} & {\scriptsize{$1.56*10^7$}} & \textbf{\scriptsize{1712/1199.78}} & {\scriptsize{*/2.71e-06}} & {\scriptsize{*/1.33e-07}} & {\scriptsize{*/2.38e-07}} & {\scriptsize{*/2.25e-07}}\tabularnewline

\hline

{\scriptsize{(300, 500,000, 0.5)}} & {\scriptsize{$9.39*10^{6}$}} & \textbf{\scriptsize{598/230.24}} & {\scriptsize{13383/5313.30}} & {\scriptsize{4355/2224.32}} & {\scriptsize{23569/7087.33}} & {\scriptsize{19019/5807.34}}
\tabularnewline

{\scriptsize{(300, 500,000, 1)}} & {\scriptsize{$9.39*10^{6}$}} & \textbf{\scriptsize{1421/561.68}} & {\scriptsize{*/8.47e-08}} & {\scriptsize{*/8.18e-08}} & {\scriptsize{*/1.47e-07}} &{\scriptsize{23125/6971.15}}
\tabularnewline

{\scriptsize{(300, 500,000, 2)}} & {\scriptsize{$9.39*10^{6}$}} & \textbf{\scriptsize{ 977/382.66}} & {\scriptsize{*/1.08e-06}} & {\scriptsize{*/4.47e-07}} &  {\scriptsize{*/2.64e-07}} & {\scriptsize{*/7.62e-08}}
\tabularnewline

\hline
{\scriptsize{(100, 1,000,000, 0.5)}} & {\scriptsize{$6.27*10^{6}$}} & \textbf{\scriptsize{ 588/178.35}} & {\scriptsize{6129/1938.83}} & {\scriptsize{2824/1126.19}} & {\scriptsize{*/4.70e-08}} & {\scriptsize{24554/6419.97}}
\tabularnewline

{\scriptsize{(100, 1,000,000, 1)}} & {\scriptsize{$6.27*10^{6}$}} & \textbf{\scriptsize{1389/431.63}} & {\scriptsize{*/7.74e-08}} & {\scriptsize{8845/3530.80}} & {\scriptsize{*/5.50e-07}} & {\scriptsize{*/2.09e-07}}
\tabularnewline

{\scriptsize{(100, 1,000,000, 2)}} & {\scriptsize{$6.27*10^{6}$}} & \textbf{\scriptsize{2099/673.07}} & {\scriptsize{*/1.93e-06}} & {\scriptsize{*/4.92e-07}} & {\scriptsize{*/5.29e-06}}  & {\scriptsize{*/3.98e-06}}
\end{tabular}
\par\end{centering}
\caption{\scriptsize Iteration counts and runtimes (in seconds) for the sparse logistic regression problem in
\S~\ref{logistic regression}. The tolerances are set to $10^{-8}$. Entries marked with * did not converge in the time limit of $7200$ seconds. The ATR metric is 14.06.\label{tab108:logL1}}
\end{table}

\subsection{Lasso}\label{Lasso}

Consider the problem
\begin{align*}
\min_{z\in \mathbb R^{n}}\  & \left[f(z):=\frac{1}{2}\|Az-b\|^2\right]\\
\text{s.t.}\  & \|z\|_{1}\leq C
\end{align*}
where $A \in \mathbb R^{m \times n}$ and $b \in \mathbb R^{m}$ and $C>0$ is a regularization parameter. For our experiments in \S~\ref{Lasso}, the matrix $A$ and the vector $b$ are taken from the linear programming test problems considered in the datasets `Meszaros' and `Mittelmann'. The regularization parameter, $C$ is taken to be either $1$, $5$, or $10$ and for each problem instance, the initial point $x_0$ is chosen to be a random point satisfying $\|x_0\|_{1}\leq C$. The Lipschitz constant $\bar L$ of the gradient of $f$ is just $\|A\|^2.$ The results comparing the five methods are presented in Tables~\ref{tab1013:Lasso} and \ref{tab1013:Lasso-2}, where the names of each problem instance considered from the Meszaros and Mittelmann datasets are also given.

\begin{table}[!tbh]
\captionsetup{font=scriptsize}
\begin{centering}
\begin{tabular}{>{\centering}p{2.25cm}>{\centering}p{1.2cm}>{\centering}p{1.2cm}|>{\centering}p{1.8cm}>{\centering}p{1.75cm}>{\centering}p{1.75cm}>{\centering}p{1.7cm}>{\centering}p{1.7cm}}
\multicolumn{1}{c}{\textbf{\scriptsize{}Parameters}} &  \multicolumn{1}{c}{\textbf{\scriptsize{}Lipschitz}}& \multicolumn{1}{c|}{\textbf{\scriptsize{}Name}} & \multicolumn{5}{c}{\textbf{\scriptsize{}Iteration Count/Runtime (seconds)}}\tabularnewline
\hline 
{\scriptsize{}($m$,$n$,$C$)} & {\scriptsize{}$\bar L$} & {\scriptsize{}} & {\scriptsize{}RPF-SFISTA} & {\scriptsize{}FISTA-BT} & {\scriptsize FISTA-R}& {\scriptsize{}RA-FISTA} & {\scriptsize{}GR-FISTA} \tabularnewline
\hline 
\hline 
{\scriptsize{(825, 8,627, 1)}}  & {\scriptsize{$1.55*10^3$}} & {\scriptsize{aa03}} & {\scriptsize{\textbf{226/0.28}}} & {\scriptsize{4888/4.86}} & {\scriptsize{4598/4.45}} & {\scriptsize{706/0.86}} & {\scriptsize{636/0.66}}\tabularnewline

{\scriptsize{(825, 8,627, 5)}}  & {\scriptsize{$1.55*10^3$}} & {\scriptsize{aa03}} & {\scriptsize{\textbf{722/0.72}}} & {\scriptsize{8482/8.07}} & {\scriptsize{8482/7.82}} & {\scriptsize{1322/1.44}} & {\scriptsize{1198/1.06}}\tabularnewline

{\scriptsize{(825, 8,627, 10)}}  & {\scriptsize{$1.55*10^3$}} & {\scriptsize{aa03}} & {\scriptsize{\textbf{472/0.52}}} & {\scriptsize{13132/13.04}} & {\scriptsize{9654/9.18}} & {\scriptsize{1352/1.47}} & {\scriptsize{1239/1.14}}\tabularnewline

\hline

% {\scriptsize{(825, 8,627, 1)}} & {\scriptsize{$1.55*10^3$}} & {\scriptsize{\textbf{226/0.24}}} & {\scriptsize{4888/4.58}} & {\scriptsize{4598/4.36}} & {\scriptsize{706/0.76}} & {\scriptsize{636/0.54}}\tabularnewline

% {\scriptsize{(825, 8,627, 5)}} & {\scriptsize{$1.55*10^3$}} & {\scriptsize{\textbf{722/0.74}}} & {\scriptsize{8482/7.94}} & {\scriptsize{8482/7.82}} & {\scriptsize{1322/1.49}} & {\scriptsize{1198/1.06}}\tabularnewline

% {\scriptsize{(825, 8,627, 10)}} & {\scriptsize{$1.5*10^3$}} & {\scriptsize{\textbf{472/0.46}}} & {\scriptsize{13132/12.67}} & {\scriptsize{9654/9.03}} & {\scriptsize{1352/1.49}} & {\scriptsize{1239/1.11}}\tabularnewline

% \hline

{\scriptsize{(426, 7,195, 1)}}  & {\scriptsize{$1.73*10^3$}} & {\scriptsize{aa4}} & {\scriptsize{\textbf{348/0.29}}} & {\scriptsize{1522/1.27}} & {\scriptsize{1461/1.22}} & {\scriptsize{833/0.83}} & {\scriptsize{735/0.56}}\tabularnewline

{\scriptsize{(426, 7,195, 5)}} & {\scriptsize{$1.73*10^3$}} & {\scriptsize{aa4}} & {\scriptsize{\textbf{432/0.44}}} & {\scriptsize{6898/5.87}} & {\scriptsize{6898/5.48}} & {\scriptsize{1589/1.44}} & {\scriptsize{1605/1.30}}\tabularnewline

{\scriptsize{(426, 7,195, 10)}} & {\scriptsize{$1.73*10^3$}} & {\scriptsize{aa4}} & {\scriptsize{\textbf{735/0.64}}} & {\scriptsize{10377/10.03}} & {\scriptsize{8503/7.39}} & {\scriptsize{2289/2.14}} & {\scriptsize{2203/1.75}}\tabularnewline

\hline

{\scriptsize{(50, 6,774, 1)}} & {\scriptsize{$1.73*10^4$}} & {\scriptsize{air02}} & {\scriptsize{\textbf{370/0.34}}} & {\scriptsize{5223/4.32}} & {\scriptsize{3570/2.66}} & {\scriptsize{3482/3.32}} & {\scriptsize{3380/2.65}}\tabularnewline

{\scriptsize{(50, 6,774, 5)}} & {\scriptsize{$1.73*10^4$}} & {\scriptsize{air02}} & {\scriptsize{\textbf{3410/2.78}}} & {\scriptsize{125979/102.89}} & {\scriptsize{124619/87.70}} & {\scriptsize{28156/23.63}} & {\scriptsize{21579/15.82}}\tabularnewline

{\scriptsize{(50, 6,774, 10)}} & {\scriptsize{$1.73*10^4$}} & {\scriptsize{air02}} & {\scriptsize{\textbf{5078/2.59}}} & {\scriptsize{129536/57.52}} & {\scriptsize{131588/61.23}} & {\scriptsize{1551/1.07}} & {\scriptsize{1353/0.59}}\tabularnewline

\hline

{\scriptsize{(124, 10,757, 1)}} & {\scriptsize{$1.57*10^4$}} & {\scriptsize{air03}} & {\scriptsize{\textbf{627/11.80}}} & {\scriptsize{7138/127.32}} & {\scriptsize{5173/78.53}} & {\scriptsize{3211/42.12}} & {\scriptsize{3488/65.76}}\tabularnewline

{\scriptsize{(124, 10,757, 5)}} & {\scriptsize{$1.57*10^4$}} & {\scriptsize{air03}} & {\scriptsize{\textbf{1090/17.61}}} & {\scriptsize{41501/618.35}} & {\scriptsize{41501/605.35}} & {\scriptsize{9257/139.24}} & {\scriptsize{10359/219.91}}\tabularnewline

{\scriptsize{(124, 10,757, 10)}} & {\scriptsize{$1.57*10^4$}} & {\scriptsize{air03}} & {\scriptsize{\textbf{2996/70.14}}} & {\scriptsize{286189/4624.86}} & {\scriptsize{289853/4549.27}} & {\scriptsize{93958/1598.77}} & {\scriptsize{30278/552.48}}\tabularnewline

\hline

{\scriptsize{(823, 8,904, 1)}} & {\scriptsize{$1.50*10^3$}} & {\scriptsize{air04}} & {\scriptsize{\textbf{339/0.32}}} & {\scriptsize{5233/5.20}} & {\scriptsize{3139/3.13}} & {\scriptsize{1150/1.20}} & {\scriptsize{980/1.02}}\tabularnewline

{\scriptsize{(823, 8,904, 5)}} & {\scriptsize{$1.50*10^3$}} & {\scriptsize{air04}} & {\scriptsize{\textbf{342/0.36}}} & {\scriptsize{5575/5.59}} & {\scriptsize{5085/5.35}} & {\scriptsize{1163/1.26}} & {\scriptsize{977/1.01}}\tabularnewline

{\scriptsize{(823, 8,904, 10)}}  & {\scriptsize{$1.50*10^3$}} & {\scriptsize{air04}} & {\scriptsize{\textbf{903/0.86}}} & {\scriptsize{10956/10.02}} & {\scriptsize{9270/8.62}} & {\scriptsize{1257/1.39}} & {\scriptsize{1342/1.36}}\tabularnewline

\hline

% {\scriptsize{(426, 7,195, 1)}} & {\scriptsize{$1.73*10^3$}} & {\scriptsize{\textbf{348/0.28}}} & {\scriptsize{1522/1.15}} & {\scriptsize{1461/1.13}} & {\scriptsize{833/0.79}} & {\scriptsize{735/0.65}}\tabularnewline

% {\scriptsize{(426, 7,195, 5)}} & {\scriptsize{$1.73*10^3$}} & {\scriptsize{\textbf{432/0.44}}} & {\scriptsize{6898/5.33}} & {\scriptsize{6898/5.00}} & {\scriptsize{1589/1.40}} & {\scriptsize{1605/1.51}}\tabularnewline

% {\scriptsize{(426, 7,195, 10)}} & {\scriptsize{$1.73*10^3$}} & {\scriptsize{\textbf{735/0.65}}} & {\scriptsize{10377/9.24}} & {\scriptsize{8503/6.97}} & {\scriptsize{2289/2.03}} & {\scriptsize{2203/2.07}}\tabularnewline

% \hline

% {\scriptsize{(825, 8,627, 1)}} & {\scriptsize{$1.55*10^3$}} & {\scriptsize{\textbf{226/0.21}}} & {\scriptsize{4888/4.89}} & {\scriptsize{4598/4.34}} & {\scriptsize{706/0.77}} & {\scriptsize{636/0.69}}\tabularnewline

% {\scriptsize{(825, 8,627, 5)}} & {\scriptsize{$1.55*10^3$}} & {\scriptsize{\textbf{722/0.68}}} & {\scriptsize{8482/8.53}} & {\scriptsize{8482/8.00}} & {\scriptsize{1322/1.47}} & {\scriptsize{1198/1.21}}\tabularnewline

% {\scriptsize{(825, 8,627, 10)}} & {\scriptsize{$1.55*10^3$}} & {\scriptsize{\textbf{472/0.51}}} & {\scriptsize{13132/13.78}} & {\scriptsize{9654/9.57}} & {\scriptsize{1352/1.50}} & {\scriptsize{1239/1.27}}\tabularnewline

% \hline

{\scriptsize{(769, 2,561, 1)}}  & {\scriptsize{$1.54*10^3$}} & {\scriptsize{gen}} & {\scriptsize{\textbf{93/0.06}}} & {\scriptsize{450/0.19}} & {\scriptsize{323/0.15}} & {\scriptsize{331/0.15}} & {\scriptsize{254/0.11}}\tabularnewline

{\scriptsize{(769, 2,561, 5)}} & {\scriptsize{$1.54*10^3$}} & {\scriptsize{gen}} & {\scriptsize{\textbf{177/0.09}}} & {\scriptsize{1488/0.66}} & {\scriptsize{948/0.43}} & {\scriptsize{498/0.20}} & {\scriptsize{454/0.18}}\tabularnewline

{\scriptsize{(769, 2,561, 10)}}  & {\scriptsize{$1.54*10^3$}} & {\scriptsize{gen}} & {\scriptsize{\textbf{190/0.11}}} & {\scriptsize{2613/1.09}} & {\scriptsize{1491/0.69}} & {\scriptsize{689/0.30}} & {\scriptsize{591/0.27}}\tabularnewline

\hline

{\scriptsize{(163, 28,016, 1)}} & {\scriptsize{$5.64*10^4$}} & {\scriptsize{us04}} & {\scriptsize{\textbf{438/7.57}}} & {\scriptsize{11162/158.37}} & {\scriptsize{9858/135.20}} & {\scriptsize{5954/99.42}} & {\scriptsize{6357/104.38}}\tabularnewline

{\scriptsize{(163, 28,016, 5)}} & {\scriptsize{$5.64*10^4$}} & {\scriptsize{us04}} & {\scriptsize{\textbf{1119/21.39}}} & {\scriptsize{41307/644.84}} & {\scriptsize{33990/520.40}} & {\scriptsize{8530/135.05}} & {\scriptsize{9609/171.61}}\tabularnewline

{\scriptsize{(163, 28,016, 10)}} & {\scriptsize{$5.64*10^4$}} & {\scriptsize{us04}} & {\scriptsize{\textbf{1713/33.49}}} & {\scriptsize{422708/6284.62}} & {\scriptsize{446469/6230.74}} & {\scriptsize{14290/215.23}} & {\scriptsize{14280/205.82}}\tabularnewline

\hline

{\scriptsize{(520, 1,544, 1)}} & {\scriptsize{$3.10*10^4$}} & {\scriptsize{rosen1}} & {\scriptsize{{30/0.009}}} & {\scriptsize{30/0.007}} & {\scriptsize{30/0.006}} & {\scriptsize{31/0.007}} & {\scriptsize{\textbf{14/0.005}}}\tabularnewline

{\scriptsize{(520, 1,544, 5)}} & {\scriptsize{$3.10*10^4$}} & {\scriptsize{rosen1}} & {\scriptsize{\textbf{87/0.02}}} & {\scriptsize{151/0.02}} & {\scriptsize{128/0.02}} & {\scriptsize{151/0.03}} & {\scriptsize{{108/0.02}}}\tabularnewline

{\scriptsize{(520, 1,544, 10)}} & {\scriptsize{$3.10*10^4$}} & {\scriptsize{rosen1}} & {\scriptsize{\textbf{101/0.02}}} & {\scriptsize{288/0.05}} & {\scriptsize{205/0.04}} & {\scriptsize{231/0.05}} & {\scriptsize{{181/0.03}}}\tabularnewline

\hline

{\scriptsize{(2,056, 6,152, 1)}} & {\scriptsize{$1.17*10^5$}} & {\scriptsize{rosen10}} & {\scriptsize{{29/0.02}}} & {\scriptsize{29/0.02}} & {\scriptsize{29/0.02}} & {\scriptsize{58/0.04}} & {\scriptsize{\textbf{20/0.02}}}\tabularnewline

{\scriptsize{(2,056, 6,152, 5)}} & {\scriptsize{$1.17*10^5$}} & {\scriptsize{rosen10}} & {\scriptsize{\textbf{163/0.13}}} & {\scriptsize{624/0.39}} & {\scriptsize{363/0.22}} & {\scriptsize{449/0.31}} & {\scriptsize{388/0.26}}\tabularnewline

{\scriptsize{(2,056, 6,152, 10)}} & {\scriptsize{$1.17*10^5$}} & {\scriptsize{rosen10}} & {\scriptsize{\textbf{209/0.16}}} & {\scriptsize{2491/1.62}} & {\scriptsize{1172/0.78}} & {\scriptsize{535/0.40}} & {\scriptsize{421/0.25}}\tabularnewline

\hline

{\scriptsize{(135, 6,469, 1)}} & {\scriptsize{$4.93*10^3$}} & {\scriptsize{crew1}} & {\scriptsize{\textbf{477/0.28}}} & {\scriptsize{9519/5.31}} & {\scriptsize{8573/5.12}} & {\scriptsize{2689/1.57}} & {\scriptsize{2821/1.67}}\tabularnewline

{\scriptsize{(135, 6,469, 5)}} & {\scriptsize{$4.93*10^3$}} & {\scriptsize{crew1}} & {\scriptsize{\textbf{1790/1.25}}} & {\scriptsize{44385/34.44}} & {\scriptsize{31153/26.13}} & {\scriptsize{4655/3.65}} & {\scriptsize{5529/4.11}}\tabularnewline

{\scriptsize{(135, 6,469, 10)}} & {\scriptsize{$4.93*10^3$}} & {\scriptsize{crew1}} & {\scriptsize{\textbf{2362/1.63}}} & {\scriptsize{1000023/718.28}} & {\scriptsize{1000023/764.45}} & {\scriptsize{9093/6.88}} & {\scriptsize{9473/6.92}}\tabularnewline

% \hline

% {\scriptsize{(769, 2,561, 1)}} & {\scriptsize{$1.54*10^3$}} & {\scriptsize{\textbf{93/0.06}}} & {\scriptsize{450/0.17}} & {\scriptsize{323/0.15}} & {\scriptsize{331/0.13}} & {\scriptsize{254/0.12}}\tabularnewline

% {\scriptsize{(769, 2,561, 5)}} & {\scriptsize{$1.54*10^3$}} & {\scriptsize{\textbf{177/0.10}}} & {\scriptsize{1488/0.64}} & {\scriptsize{948/0.44}} & {\scriptsize{498/0.20}} & {\scriptsize{454/0.19}}\tabularnewline

% {\scriptsize{(769, 2,561, 10)}} & {\scriptsize{$1.54*10^3$}} & {\scriptsize{\textbf{190/0.10}}} & {\scriptsize{2613/1.29}} & {\scriptsize{1491/0.75}} & {\scriptsize{689/0.31}} & {\scriptsize{591/0.28}}\tabularnewline

\end{tabular}
\par\end{centering}
\caption{\scriptsize Iteration counts and runtimes (in seconds) for the Lasso problem in
\S~\ref{Lasso}. The tolerances are set to $10^{-13}$. Entries marked with * did not converge in the time limit of $7200$ seconds. The ATR metric is 3.87.\label{tab1013:Lasso}}
\end{table}

\begin{table}[!tbh]
\captionsetup{font=scriptsize}
\begin{centering}
\begin{tabular}{>{\centering}p{2.8cm}>{\centering}p{1.13cm}>{\centering}p{1.13cm}|>{\centering}p{1.768cm}>{\centering}p{1.73cm}>{\centering}p{1.7cm}>{\centering}p{1.6cm}>{\centering}p{1.55cm}}
\multicolumn{1}{c}{\textbf{\scriptsize{}Parameters}} & \multicolumn{1}{c}{\textbf{\scriptsize{}Lipschitz}}& \multicolumn{1}{c|}{\textbf{\scriptsize{}Name}} & \multicolumn{5}{c}{\textbf{\scriptsize{}Iteration Count/Runtime (seconds)}}\tabularnewline
\hline 
{\scriptsize{}($m$,$n$,$C$)} & {\scriptsize{}$\bar L$} & {\scriptsize{}} & {\scriptsize{}RPF-SFISTA} & {\scriptsize{}FISTA-BT} & {\scriptsize FISTA-R}& {\scriptsize{}RA-FISTA} & {\scriptsize{}GR-FISTA} \tabularnewline
\hline 
\hline

{\scriptsize{(905, 1,513, 1)}} & {\scriptsize{$2.14*10^5$}} & {\scriptsize{cr42}} & {\scriptsize{{161/0.05}}} & {\scriptsize{451/0.07}} & {\scriptsize{408/0.07}} & {\scriptsize{81/0.01}} & {\scriptsize{\textbf{83/0.01}}}\tabularnewline

{\scriptsize{(905, 1,513, 5)}} & {\scriptsize{$2.14*10^5$}} &  {\scriptsize{cr42}} & {\scriptsize{{765/0.15}}} & {\scriptsize{11628/1.89}} & {\scriptsize{4503/0.74}} & {\scriptsize{627/0.09}} & {\scriptsize{\textbf{395/0.06}}}\tabularnewline

{\scriptsize{(905, 1,513, 10)}} & {\scriptsize{$2.14*10^5$}} &  {\scriptsize{cr42}} & {\scriptsize{{8180/1.47}}} & {\scriptsize{197048/39.93}} & {\scriptsize{39202/8.13}} & {\scriptsize{4027/0.77}} & {\scriptsize{\textbf{3465/0.62}}}\tabularnewline

\hline

{\scriptsize{(71, 36,699, 1)}} & {\scriptsize{$2.49*10^4$}} &  {\scriptsize{kl02}} & {\scriptsize{\textbf{2585/75.80}}} & {\scriptsize{8633/188.86}} & {\scriptsize{8633/192.25}} & {\scriptsize{5812/139.33}} & {\scriptsize{5966/138.34}}\tabularnewline

{\scriptsize{(71, 36,699, 5)}} & {\scriptsize{$2.49*10^4$}} &  {\scriptsize{kl02}} & {\scriptsize{{659/21.54}}} & {\scriptsize{1574/39.99}} & {\scriptsize{*/2.99e-07}} & {\scriptsize{459/10.27}} & {\scriptsize{\textbf{412/8.59}}}\tabularnewline

{\scriptsize{(71, 36,699, 10)}} & {\scriptsize{$2.49*10^4$}} &  {\scriptsize{kl02}} & {\scriptsize{{500/11.22}}} & {\scriptsize{1571/27.56}} & {\scriptsize{811/22.23}} & {\scriptsize{374/6.49}} & {\scriptsize{\textbf{345/6.30}}}\tabularnewline

\hline

{\scriptsize{(664, 46,915, 1)}} & {\scriptsize{$1.58*10^4$}} &  {\scriptsize{t0331-4l}} & {\scriptsize{\textbf{499/10.67}}} & {\scriptsize{4742/94.66}} & {\scriptsize{3291/85.86}} & {\scriptsize{1988/43.69}} & {\scriptsize{1963/35.04}}\tabularnewline

{\scriptsize{(664, 46,915, 5)}} & {\scriptsize{$1.58*10^4$}} &  {\scriptsize{t0331-4l}} & {\scriptsize{\textbf{854/20.36}}} & {\scriptsize{22202/440.41}} & {\scriptsize{15215/391.61}} & {\scriptsize{3726/87.56}} & {\scriptsize{4589/80.62}}\tabularnewline

{\scriptsize{(664, 46,915, 10)}} & {\scriptsize{$1.58*10^4$}} &  {\scriptsize{t0331-4l}} & {\scriptsize{\textbf{598/13.81}}} & {\scriptsize{115672/2826.81}} & {\scriptsize{91697/2846.42}} & {\scriptsize{3959/82.91}} & {\scriptsize{3846/77.28}}\tabularnewline

\hline

{\scriptsize{(507, 63,516, 1)}} & {\scriptsize{$2.23*10^4$}} &  {\scriptsize{rail507}} & {\scriptsize{\textbf{441/9.45}}} & {\scriptsize{16380/322.10}} & {\scriptsize{14546/379.22}} & {\scriptsize{3871/89.02}} & {\scriptsize{3944/79.39}}\tabularnewline

{\scriptsize{(507, 63,516, 5)}} & {\scriptsize{$2.23*10^4$}} &  {\scriptsize{rail507}} & {\scriptsize{\textbf{633/16.98}}} & {\scriptsize{33798/946.03}} & {\scriptsize{25542/801.14}} & {\scriptsize{5445/150.25}} & {\scriptsize{5724/155.04}}\tabularnewline

{\scriptsize{(507, 63,516, 10)}} & {\scriptsize{$2.23*10^4$}} &  {\scriptsize{rail507}} & {\scriptsize{\textbf{1081/26.66}}} & {\scriptsize{54412/1198.24}} & {\scriptsize{42622/1202.48}} & {\scriptsize{6682/167.30}} & {\scriptsize{7649/179.36}}\tabularnewline

\hline

{\scriptsize{(516, 47,827, 1)}} & {\scriptsize{$2.07*10^4$}} &  {\scriptsize{rail516}} & {\scriptsize{\textbf{1132/24.25}}} & {\scriptsize{5795/96.74}} & {\scriptsize{5795/131.08}} & {\scriptsize{3890/76.33}} & {\scriptsize{3719/91.63}}\tabularnewline

{\scriptsize{(516, 47,827, 5)}} & {\scriptsize{$2.07*10^4$}} &  {\scriptsize{rail516}} & {\scriptsize{\textbf{615/14.33}}} & {\scriptsize{16843/318.72}} & {\scriptsize{12423/295.14}} & {\scriptsize{5053/129.64}} & {\scriptsize{5188/154.17}}\tabularnewline

{\scriptsize{(516, 47,827, 10)}} & {\scriptsize{$2.07*10^4$}} &  {\scriptsize{rail516}} & {\scriptsize{\textbf{1053/26.65}}} & {\scriptsize{25057/563.01}} & {\scriptsize{25057/720.52}} & {\scriptsize{6155/163.21}} & {\scriptsize{5688/155.14}}\tabularnewline

\hline

{\scriptsize{(582, 56,097, 1)}} & {\scriptsize{$3.46*10^4$}} &  {\scriptsize{rail582}} & {\scriptsize{\textbf{656/16.51}}} & {\scriptsize{10000/263.01}} & {\scriptsize{10000/327.45}} & {\scriptsize{7427/196.50}} & {\scriptsize{7139/218.76}}\tabularnewline

{\scriptsize{(582, 56,097, 5)}} & {\scriptsize{$3.46*10^4$}} &  {\scriptsize{rail582}} & {\scriptsize{\textbf{518/14.02}}} & {\scriptsize{*/4.06e-12}} & {\scriptsize{*/5.35e-12}} & {\scriptsize{7943/210.25}} & {\scriptsize{7978/251.48}}\tabularnewline

{\scriptsize{(582, 56,097, 10)}} & {\scriptsize{$3.46*10^4$}} &  {\scriptsize{rail582}} & {\scriptsize{\textbf{1255/36.54}}} & {\scriptsize{*/4.23e-11}} & {\scriptsize{*/2.60e-11}} & {\scriptsize{8036/195.24}} & {\scriptsize{8278/248.95}}\tabularnewline

\hline 

{\scriptsize{(2,586, 923,269, 1)}} & {\scriptsize{$2.46*10^{5}$}} &  {\scriptsize{rail2586}} & {\scriptsize{\textbf{3254/343.08}}} & {\scriptsize{*/5.98e-11}} & {\scriptsize{*/8.33e-11}} & {\scriptsize{24725/2631.23}} & {\scriptsize{25618/2729.25}}\tabularnewline

{\scriptsize{(2,586, 923,269, 5)}} & {\scriptsize{$2.46*10^{5}$}} &  {\scriptsize{rail2586}} & {\scriptsize{\textbf{2039/203.52}}} & {\scriptsize{*/2.21e-10}} & {\scriptsize{*/2.42e-10}} & {\scriptsize{*/6.14e-11}} & {\scriptsize{26307/2614.78}}\tabularnewline

{\scriptsize{(2,586, 923,269, 10)}} & {\scriptsize{$2.46*10^{5}$}} &  {\scriptsize{rail2586}} & {\scriptsize{\textbf{1843/163.57}}} & {\scriptsize{*/1.07e-09}} & {\scriptsize{*/1.22e-09}} & {\scriptsize{26475/2630.74}} & {\scriptsize{31177/3194.57}}\tabularnewline

\hline 

{\scriptsize{(4,284, 1,096,894, 1)}} & {\scriptsize{$1.60*10^{5}$}} &  {\scriptsize{rail4284}} & {\scriptsize{\textbf{4138/523.35}}} & {\scriptsize{*/1.61e-11}} & {\scriptsize{*/1.62e-11}} & {\scriptsize{30633/3974.30}} & {\scriptsize{30391/3909.32}}\tabularnewline

{\scriptsize{(4,284, 1,096,894, 5)}} & {\scriptsize{$1.60*10^{5}$}} &  {\scriptsize{rail4284}} & {\scriptsize{\textbf{3852/454.55}}} & {\scriptsize{*/5.05e-11}} & {\scriptsize{*/5.66e-11}} & {\scriptsize{*/7.70e-08}} & {\scriptsize{*/7.53e-08}}\tabularnewline

{\scriptsize{(4,284, 1,096,894, 10)}} & {\scriptsize{$1.60*10^{5}$}} &  {\scriptsize{rail4284}} & {\scriptsize{\textbf{2970/340.49}}} & {\scriptsize{*/1.06e-09}} & {\scriptsize{*/1.15e-09}} & {\scriptsize{*/6.86e-11}} & {\scriptsize{31325/3394.41}}\tabularnewline

\hline

{\scriptsize{(73, 123,409, 1)}} & {\scriptsize{$2.56*10^{5}$}} &  {\scriptsize{nw14}} & {\scriptsize{\textbf{793/28.11}}} & {\scriptsize{4828/167.18}} & {\scriptsize{3416/114.27}} & {\scriptsize{4718/202.63}} & {\scriptsize{4527/178.69}}\tabularnewline

{\scriptsize{(73, 123,409, 5)}} & {\scriptsize{$2.56*10^{5}$}} &  {\scriptsize{nw14}} & {\scriptsize{\textbf{4753/186.80}}} & {\scriptsize{96469/3317.70}} & {\scriptsize{80492/2513.14}} & {\scriptsize{10451/341.44}} & {\scriptsize{9860/310.00}}\tabularnewline

{\scriptsize{(73, 123,409, 10)}} & {\scriptsize{$2.56*10^{5}$}} &  {\scriptsize{nw14}} & {\scriptsize{\textbf{2460/92.08}}} & {\scriptsize{120336/3364.65}} & {\scriptsize{86166/2366.60}} & {\scriptsize{18593/604.72}} & {\scriptsize{18294/593.66}}\tabularnewline

\hline

{\scriptsize{(27,441, 30,733, 1)}} & {\scriptsize{$3.6*10^{11}$}} &  {\scriptsize{baxter}} & {\scriptsize{\textbf{2431/35.26}}} & {\scriptsize{4671/70.73}} & {\scriptsize{4276/65.40}} & {\scriptsize{5098/67.13}} & {\scriptsize{4339/44.83}}\tabularnewline

{\scriptsize{(27,441, 30,733, 5)}} & {\scriptsize{$3.6*10^{11}$}} &  {\scriptsize{baxter}} & {\scriptsize{{11805/140.86}}} & {\scriptsize{12352/186.97}} & {\scriptsize{9150/134.82}} & {\scriptsize{11976/152.81}} & {\scriptsize{\textbf{8329/82.57}}}\tabularnewline

{\scriptsize{(27,441, 30,733, 10)}} & {\scriptsize{$3.6*10^{11}$}} &  {\scriptsize{baxter}} & {\scriptsize{{20929/240.32}}} & {\scriptsize{21569/322.11}} & {\scriptsize{19936/307.04}} & {\scriptsize{16833/210.20}} & {\scriptsize{\textbf{12231/122.60}}}\tabularnewline

\end{tabular}
\par\end{centering}
\caption{\scriptsize Iteration counts and runtimes (in seconds) for the Lasso Problem in
\S~\ref{Lasso}. The tolerances are set to $10^{-13}$. Entries marked with * did not converge in the time limit of $7200$ seconds. The ATR metric is 6.32.\label{tab1013:Lasso-2}}
\end{table}

\subsection{Dense QP with Simplex Constraints}\label{subsec:nonconvex_qpsimplex}
Given a pair of dimensions $(m,n)\in\mathbb{N}^{2}$, a scalar
pair $(\tau_1, \tau_2)\in\R_{++}^{2}$, matrices $B\in\r^{n\times n}$ and $C\in\R^{m\times n}$, positive diagonal matrix $D\in\R^{n\times n}$,
and a vector $d\in\R^{m}$, this subsection considers
the problem 
\begin{align*}
\min_{z}\  & \left[f(z):=\frac{\tau_{1}}{2}\|DBz\|^{2}+\frac{\tau_{2}}{2}\|{ C}z-d\|^{2}\right]\\
\text{s.t.}\  & z\in\Delta^{n},
\end{align*}
where $\Delta^{n}:=\{x\in\ \Re_+^{n}:\sum_{i=1}^{n}x_{i}=1\}$.
For our experiments in \S~\ref{subsec:nonconvex_qpsimplex}, we vary the  dimensions and generate the matrices $B$ and $C$ to be fully dense. The entries
of $B$, $C$, and $d$ (resp. $D$) are generated by sampling
from the uniform distribution ${\cal U}[0,1]$ (resp. ${\cal U}[1,\alpha]$) where the parameter $\alpha\geq 1$ is varied. The initial starting point $x_{0}$
is generated as $\hat x/\sum_{i=1}^{n}\hat x_{i}$, where the
entries of $\hat x$ are sampled from the ${\cal U}[0,1]$ distribution. Finally, we choose $(\tau_1, \tau_2)\in\R_{++}^{2}$ so that $\bar L=\lam_{\max}(\nabla^{2}f)$
and $\bar \mu=\lam_{\min}(\nabla^{2}f)$ are the various values given in the tables of this subsection. The results comparing the five methods are presented in Tables~\ref{tab108:qpsimplex} and \ref{tab1013:qpsimplex}. Table~\ref{tab108:qpsimplex} and Table~\ref{tab1013:qpsimplex} present the performance of all five methods on the same exact 12 instances, but for target accuracies $\epsilon=10^{-8}$ and $\epsilon=10^{-13}$, respectively.

\begin{table}[!tbh]
\captionsetup{font=scriptsize}
\begin{centering}
\begin{tabular}{>{\centering}p{1.75cm}>{\centering}p{1.8cm}|>{\centering}p{1.8cm}>{\centering}p{1.8cm}>{\centering}p{1.8cm}>{\centering}p{1.8cm}>{\centering}p{1.8cm}}
\multicolumn{1}{c}{\textbf{\scriptsize{}Dimensions}} & \multicolumn{1}{c|}{\textbf{\scriptsize{}Curvatures}} & \multicolumn{5}{c}{\textbf{\scriptsize{}Iteration Count/Runtime (seconds)}}\tabularnewline
\hline 
{\scriptsize{}($m$,$n$)} & {\scriptsize{}($\bar \mu$,$\bar L$)} & {\scriptsize{}RPF-SFISTA} & {\scriptsize{}FISTA-BT} & {\scriptsize FISTA-R}& {\scriptsize{}RA-FISTA} & {\scriptsize{}GR-FISTA} \tabularnewline
\hline 
\hline 
{\scriptsize{(1000, 5000)}} & {\scriptsize{$(10^{-8},10^2)$}} & {\scriptsize{\textbf{760/1261.82}}} & {\scriptsize{*/1.13e-07}} & {\scriptsize{*/8.91e-08}} & {\scriptsize{2385/3050.66}} & {\scriptsize{1929/2460.67}}\tabularnewline

{\scriptsize{(1000, 5000)}} & {\scriptsize{$(10^{-6},10^2)$}} & \textbf{\scriptsize{322/536.82}} & {\scriptsize{3201/4345.55}} & {\scriptsize{3201/4529.81}} & {\scriptsize{2291/2924.71}} & {\scriptsize{2163/2781.60}} \tabularnewline

{\scriptsize{(1000, 5000)}} & {\scriptsize{$(10^{-4},10^3)$}} & \textbf{\scriptsize{140/224.95}} & {\scriptsize{619/886.74}} & {\scriptsize{619/886.54}} & {\scriptsize{2090/2685.77}} & {\scriptsize{1714/2206.25}}\tabularnewline

{\scriptsize{(1000, 5000)}} & {\scriptsize{$(10^{-6},10^{3})$}} & \textbf{\scriptsize{113/181.25}} & {\scriptsize{623/991.30}} & {\scriptsize{623/854.16}} & {\scriptsize{2569/3258.33}} & {\scriptsize{2593/3509.69}}\tabularnewline

{\scriptsize{(1000, 5000)}} & {\scriptsize{$(10^{-7},10^{4})$}} & \textbf{\scriptsize{88/141.95}} & {\scriptsize{357/585.77}} & {\scriptsize{357/534.64}} & {\scriptsize{2840/3649.55}} & {\scriptsize{2719/3512.51}}\tabularnewline

{\scriptsize{(1000, 5000)}} & {\scriptsize{$(10^{-4},10^{6})$}} & \textbf{\scriptsize{72/108.84}} & {\scriptsize{131/197.84}} & {\scriptsize{102/151.87}} & {\scriptsize{3231/4136.42}} & {\scriptsize{2735/3483.86}}\tabularnewline

\hline

{\scriptsize{(2000, 10000)}} & {\scriptsize{$(10^{-4},10^{4})$}} & \textbf{\scriptsize{120/1162.83}} & {\scriptsize{249/2413.96}} & {\scriptsize{249/2549.80}} & {\scriptsize{*/7.41e-05}} & {\scriptsize{*/2.67e-05}}
\tabularnewline

{\scriptsize{(2000, 10000)}} & {\scriptsize{$(10^{-4},10^{4})$}} & \textbf{\scriptsize{69/694.64}} & {\scriptsize{128/1257.82}} & {\scriptsize{83/823.70}} & {\scriptsize{*/7.91e-05}} &{\scriptsize{*/2.77e-05}}
\tabularnewline

{\scriptsize{(2000, 10000)}} & {\scriptsize{$(10^{-4},10^{4})$}} & \textbf{\scriptsize{75/735.75}} & {\scriptsize{131/1287.77}} & {\scriptsize{97/965.95}} &  {\scriptsize{*/3.20e-02}} & {\scriptsize{*/2.76e-03}}
\tabularnewline

{\scriptsize{(2000, 10000)}} & {\scriptsize{$(10^{-4},10^{6})$}} & \textbf{\scriptsize{57/540.71}} & {\scriptsize{89/820.29}} & {\scriptsize{74/716.82}} & {\scriptsize{*/3.93e-03}} & {\scriptsize{*/2.27e-03}}
\tabularnewline

{\scriptsize{(2000, 10000)}} & {\scriptsize{$(10^{-4},10^{3})$}} & \textbf{\scriptsize{191/1937.20}} & {\scriptsize{685/6795.16}} & {\scriptsize{685/6895.35}} & {\scriptsize{*/6.57e-05}} & {\scriptsize{*/2.93e-05}}
\tabularnewline

{\scriptsize{(2000, 10000)}} & {\scriptsize{$(10^{-4},10^{4})$}} & \textbf{\scriptsize{92/941.71}} & {\scriptsize{249/2406.88}} & {\scriptsize{156/1563.07}} & {\scriptsize{*/3.51e-03}}  & {\scriptsize{*/1.14e-03}}
\end{tabular}
\par\end{centering}
\caption{\scriptsize Iteration counts and runtimes (in seconds) for the Vector QP with Simplex Constraints
in \S~\ref{subsec:nonconvex_qpsimplex}. The tolerances are set to $10^{-8}$. Entries marked with * did not converge in the time limit of $7200$ seconds. The ATR metric is 3.27.\label{tab108:qpsimplex}}
\end{table}

\begin{table}[!tbh]
\captionsetup{font=scriptsize}
\begin{centering}
\begin{tabular}{>{\centering}p{1.75cm}>{\centering}p{1.8cm}|>{\centering}p{1.8cm}>{\centering}p{1.8cm}>{\centering}p{1.8cm}>{\centering}p{1.8cm}>{\centering}p{1.8cm}}
\multicolumn{1}{c}{\textbf{\scriptsize{}Dimensions}} & \multicolumn{1}{c|}{\textbf{\scriptsize{}Curvatures}} & \multicolumn{5}{c}{\textbf{\scriptsize{}Iteration Count/Runtime (seconds)}}\tabularnewline
\hline 
{\scriptsize{}($m$,$n$)} & {\scriptsize{}($\bar \mu$,$\bar L$)} & {\scriptsize{}RPF-SFISTA} & {\scriptsize{}FISTA-BT} & {\scriptsize FISTA-R}& {\scriptsize{}RA-FISTA} & {\scriptsize{}GR-FISTA} \tabularnewline
\hline 
\hline 
{\scriptsize{(1000, 5000)}} & {\scriptsize{$(10^{-8},10^2)$}} & {\scriptsize{\textbf{1335/1900.39}}} & {\scriptsize{*/7.91e-08}} & {\scriptsize{*/1.02e-07}} & {\scriptsize{4009/5030.54}} & {\scriptsize{3591/4570.78}}\tabularnewline

{\scriptsize{(1000, 5000)}} & {\scriptsize{$(10^{-6},10^2)$}} & \textbf{\scriptsize{533/847.84}} & {\scriptsize{*/6.08e-10}} & {\scriptsize{*/1.10e-09}} & {\scriptsize{4037/5120.40}} & {\scriptsize{3627/4674.44}} \tabularnewline

{\scriptsize{(1000, 5000)}} & {\scriptsize{$(10^{-4},10^3)$}} & \textbf{\scriptsize{221/368.37}} & {\scriptsize{1971/2643.61}} & {\scriptsize{1971/2807.98}} & {\scriptsize{3860/4904.65}} & {\scriptsize{3320/4256.82}}\tabularnewline

{\scriptsize{(1000, 5000)}} & {\scriptsize{$(10^{-6},10^{3})$}} & \textbf{\scriptsize{169/292.13}} & {\scriptsize{1702/2445.07}} & {\scriptsize{1702/2530.27}} & {\scriptsize{4606/5894.16}} & {\scriptsize{4000/5128.65}}\tabularnewline

{\scriptsize{(1000, 5000)}} & {\scriptsize{$(10^{-7},10^{4})$}} & \textbf{\scriptsize{130/226.82}} & {\scriptsize{902/1473.35}} & {\scriptsize{902/1431.87}} & {\scriptsize{4704/6004.52}} & {\scriptsize{4187/5341.55}}\tabularnewline

{\scriptsize{(1000, 5000)}} & {\scriptsize{$(10^{-4},10^{6})$}} & \textbf{\scriptsize{100/157.56}} & {\scriptsize{274/457.90}} & {\scriptsize{248/377.36}} & {\scriptsize{4798/6149.02}} & {\scriptsize{4193/5376.22}}\tabularnewline

\hline

{\scriptsize{(2000, 10000)}} & {\scriptsize{$(10^{-4},10^{4})$}} & \textbf{\scriptsize{178/1808.51}} & {\scriptsize{*/8.52e-13}} & {\scriptsize{*/9.83e-13}} & {\scriptsize{*/7.79e-05}} & {\scriptsize{*/2.53e-05}}
\tabularnewline

{\scriptsize{(2000, 10000)}} & {\scriptsize{$(10^{-4},10^{4})$}} & \textbf{\scriptsize{133/1336.16}} & {\scriptsize{394/3951.58}} & {\scriptsize{335/3365.69}} & {\scriptsize{*/2.34e-04}} &{\scriptsize{*/2.96e-05}}
\tabularnewline

{\scriptsize{(2000, 10000)}} & {\scriptsize{$(10^{-4},10^{4})$}} & \textbf{\scriptsize{99/995.16}} & {\scriptsize{310/3114.84}} & {\scriptsize{273/2759.90}} &  {\scriptsize{*/4.56e-03}} & {\scriptsize{*/2.74e-03}}
\tabularnewline

{\scriptsize{(2000, 10000)}} & {\scriptsize{$(10^{-4},10^{6})$}} & \textbf{\scriptsize{96/907.03}} & {\scriptsize{191/1849.46}} & {\scriptsize{178/1705.75}} & {\scriptsize{*/4.13e-03}} & {\scriptsize{*/2.25e-03}}
\tabularnewline

{\scriptsize{(2000, 10000)}} & {\scriptsize{$(10^{-4},10^{3})$}} & \textbf{\scriptsize{301/2996.59}} & {\scriptsize{*/8.09e-09}} & {\scriptsize{*/6.79e-09}} & {\scriptsize{*/8.50e-04}} & {\scriptsize{*/2.78e-05}}
\tabularnewline

{\scriptsize{(2000, 10000)}} & {\scriptsize{$(10^{-4},10^{4})$}} & \textbf{\scriptsize{142/1416.76}} & {\scriptsize{*/3.29e-13}} & {\scriptsize{601/6022.82}} & {\scriptsize{*/2.34e-04}}  & {\scriptsize{*/9.04e-05}}
\end{tabular}
\par\end{centering}
\caption{\scriptsize Iteration counts and runtimes (in seconds) for the Vector QP problem with Simplex Constraints in
\S~\ref{subsec:nonconvex_qpsimplex}. The tolerances are set to $10^{-13}$. Entries marked with * did not converge in the time limit of $7200$ seconds. The ATR metric is 4.59.\label{tab1013:qpsimplex}}
\end{table}

\subsection{Dense QP with Box Constraints}\label{subsec:strconvex_qpbox}
Given a pair of dimensions $(m,n)\in\mathbb{N}^{2}$, a scalar
triple $(r,\tau_1, \tau_2)\in\R_{++}^{3}$, matrices $B\in\r^{n\times n}$ and $C\in\R^{m\times n}$, positive diagonal matrix $D\in\R^{n\times n}$,
a vector pair $(a,d)\in\R^{n}\times\R^{m}$, and a scalar $b \in \R$, this subsection considers
the problem 
\begin{align*}
\min_{z}\  & \left[f(z):=\frac{\tau_{1}}{2}\|DBz\|^{2}+\frac{\tau_{2}}{2}\|{ C}z-d\|^{2}\right]\\
\text{s.t.}\ 
& a^{T}z=b\\
& -r\leq z_i \leq r, \quad i\in\{1,...,n\}.
\end{align*}

For our experiments in \S~\ref{subsec:strconvex_qpbox}, we vary the dimensions $(m,n)$ and generate the matrices $B$ and $C$ to be fully dense. The entries
of $B$, $C$, and $d$ (resp. $D$) are generated by sampling
from the uniform distribution ${\cal U}[0,1]$ (resp. ${\cal U}[1,1000]$). The vector $a$ is varied to be either the vector that takes value $-1$ in its last component and value $1$ in all its other components or the vector that takes value $-1$ in its last 10 components and value $1$ in all its other components. The scalars $b$ and $r$ are taken to be $0$ and $5$, respectively. The initial starting point $x_{0}$
is generated as a random vector in ${\cal U}[-r,r]^{n}$. Finally, we choose $(\tau_1, \tau_2)\in\R_{++}^{2}$ so that $\bar L=\lam_{\max}(\nabla^{2}f)$
and $\bar \mu=\lam_{\min}(\nabla^{2}f)$ are the various values given in the tables of this subsection. Table~\ref{tab108:qpbox} and Table~\ref{tab1013:qpbox} present the performance of all five methods on the same exact 12 instances, but for target accuracies $\epsilon=10^{-8}$ and $\epsilon=10^{-13}$, respectively.

\begin{table}[!tbh]
\captionsetup{font=scriptsize}
\begin{centering}
\begin{tabular}{>{\centering}p{1.75cm}>{\centering}p{1.8cm}|>{\centering}p{1.8cm}>{\centering}p{1.8cm}>{\centering}p{1.8cm}>{\centering}p{1.8cm}>{\centering}p{1.8cm}}
\multicolumn{1}{c}{\textbf{\scriptsize{}Dimensions}} & \multicolumn{1}{c|}{\textbf{\scriptsize{}Curvatures}} & \multicolumn{5}{c}{\textbf{\scriptsize{}Iteration Count/Runtime (seconds)}}\tabularnewline
\hline 
{\scriptsize{}($m$,$n$)} & {\scriptsize{}($\bar \mu$,$\bar L$)} & {\scriptsize{}RPF-SFISTA} & {\scriptsize{}FISTA-BT} & {\scriptsize FISTA-R}& {\scriptsize{}RA-FISTA} & {\scriptsize{}GR-FISTA} \tabularnewline
\hline 
\hline 
{\scriptsize{(500, 1000)}} & {\scriptsize{$(10^{-4},10^2)$}} & {\scriptsize{\textbf{1035/62.34}}} & {\scriptsize{32483/1624.71}} & {\scriptsize{32483/1634.26}} & {\scriptsize{10568/443.47}} & {\scriptsize{7860/317.06}}\tabularnewline

{\scriptsize{(500, 1000)}} & {\scriptsize{$(10^{-4},10^2)$}} & \textbf{\scriptsize{2922/165.28}} & {\scriptsize{53356/2604.41}} & {\scriptsize{53356/2701.43}} & {\scriptsize{10619/441.68}} & {\scriptsize{7863/326.54}} \tabularnewline

{\scriptsize{(500, 1000)}} & {\scriptsize{$(10^{-2},10^4)$}} & \textbf{\scriptsize{607/34.03}} & {\scriptsize{8643/420.96}} & {\scriptsize{8643/436.37}} & {\scriptsize{9698/416.20}} & {\scriptsize{6820/284.80}}\tabularnewline

{\scriptsize{(500, 1000)}} & {\scriptsize{$(10^{-2},10^{4})$}} & \textbf{\scriptsize{1886/106.38}} & {\scriptsize{48990/2392.33}} & {\scriptsize{48990/2475.34}} & {\scriptsize{12016/528.72}} & {\scriptsize{6847/285.02}}\tabularnewline

{\scriptsize{(500, 1000)}} & {\scriptsize{$(10^{-3},10^{3})$}} & \textbf{\scriptsize{900/51.24}} & {\scriptsize{10023/489.64}} & {\scriptsize{10023/507.00}} & {\scriptsize{10584/461.42}} & {\scriptsize{8747/363.64}}\tabularnewline

{\scriptsize{(500, 1000)}} & {\scriptsize{$(10^{-3},10^{3})$}} & \textbf{\scriptsize{2948/166.56}} & {\scriptsize{52019/2550.12}} & {\scriptsize{52019/1444.09}} & {\scriptsize{10640/450.38}} & {\scriptsize{7901/324.36}}\tabularnewline

\hline

{\scriptsize{(1000, 2000)}} & {\scriptsize{$(10^{-2},10^{3})$}} & \textbf{\scriptsize{577/113.21}} & {\scriptsize{7507/1395.54}} & {\scriptsize{7507/1444.09}} & {\scriptsize{14361/2453.04}} & {\scriptsize{9767/1640.84}}
\tabularnewline

{\scriptsize{(1000, 2000)}} & {\scriptsize{$(10^{-2},10^{3})$}} & \textbf{\scriptsize{2030/382.98}} & {\scriptsize{*/1.37e-08}} & {\scriptsize{*/1.47e-08}} & {\scriptsize{14326/2452.20}} &{\scriptsize{9707/1633.46}}
\tabularnewline

{\scriptsize{(1000, 2000)}} & {\scriptsize{$(10^{-1},10^{4})$}} & \textbf{\scriptsize{624/113.08}} & {\scriptsize{7121/1317.06}} & {\scriptsize{7121/1367.63}} &  {\scriptsize{14363/2456.39}} & {\scriptsize{9768/1642.85}}
\tabularnewline

{\scriptsize{(1000, 2000)}} & {\scriptsize{$(10^{-1},10^{4})$}} & \textbf{\scriptsize{2054/376.27}} & {\scriptsize{*/1.91e-08}} & {\scriptsize{*/2.10e-08}} & {\scriptsize{14327/2448.34}} & {\scriptsize{9708/1634.13}}
\tabularnewline

{\scriptsize{(1000, 2000)}} & {\scriptsize{$(10^{-1},10^{5})$}} & \textbf{\scriptsize{586/105.60}} & {\scriptsize{6919/1280.26}} & {\scriptsize{6919/1327.11}} & {\scriptsize{13962/2380.55}} & {\scriptsize{9940/1674.29}}
\tabularnewline

{\scriptsize{(1000, 2000)}} & {\scriptsize{$(10^{-1},10^{5})$}} & \textbf{\scriptsize{2009/366.87}} & {\scriptsize{*/1.94e-08}} & {\scriptsize{*/1.99e-08}} & {\scriptsize{13959/2385.47}}  & {\scriptsize{9327/1572.13}}
\end{tabular}
\par\end{centering}
\caption{\scriptsize Iteration counts and runtimes (in seconds) for the Vector QP
problem with Box Constraints in \S~\ref{subsec:strconvex_qpbox}. The tolerances are set to $10^{-8}$. Entries marked with * did not converge in the time limit of $7200$ seconds. The ATR metric is 7.08.\label{tab108:qpbox}}
\end{table}

\begin{table}[!tbh]
\captionsetup{font=scriptsize}
\begin{centering}
\begin{tabular}{>{\centering}p{1.75cm}>{\centering}p{1.8cm}|>{\centering}p{1.8cm}>{\centering}p{1.8cm}>{\centering}p{1.8cm}>{\centering}p{1.8cm}>{\centering}p{1.8cm}}
\multicolumn{1}{c}{\textbf{\scriptsize{}Dimensions}} & \multicolumn{1}{c|}{\textbf{\scriptsize{}Curvatures}} & \multicolumn{5}{c}{\textbf{\scriptsize{}Iteration Count/Runtime (seconds)}}\tabularnewline
\hline 
{\scriptsize{}($m$,$n$)} & {\scriptsize{}($\bar \mu$,$\bar L$)} & {\scriptsize{}RPF-SFISTA} & {\scriptsize{}FISTA-BT} & {\scriptsize FISTA-R}& {\scriptsize{}RA-FISTA} & {\scriptsize{}GR-FISTA} \tabularnewline
\hline 
\hline 
{\scriptsize{(500, 1000)}} & {\scriptsize{$(10^{-4},10^2)$}} & {\scriptsize{\textbf{1970/90.50}}} & {\scriptsize{*/1.81e-11}} & {\scriptsize{*/3.56e-11}} & {\scriptsize{17702/741.39}} & {\scriptsize{16033/677.15}}\tabularnewline

{\scriptsize{(500, 1000)}} & {\scriptsize{$(10^{-4},10^2)$}} & \textbf{\scriptsize{6028/279.06}} & {\scriptsize{*/3.33e-10}} & {\scriptsize{*/4.60e-10}} & {\scriptsize{18304/756.50}} & {\scriptsize{15284/632.51}} \tabularnewline

{\scriptsize{(500, 1000)}} & {\scriptsize{$(10^{-2},10^4)$}} & \textbf{\scriptsize{1024/45.91}} & {\scriptsize{42950/2096.65}} & {\scriptsize{42950/2180.87}} & {\scriptsize{15453/644.29}} & {\scriptsize{12308/518.62}}\tabularnewline

{\scriptsize{(500, 1000)}} & {\scriptsize{$(10^{-2},10^{4})$}} & \textbf{\scriptsize{3074/140.47}} & {\scriptsize{*/1.32e-10}} & {\scriptsize{*/1.43e-10}} & {\scriptsize{17664/734.98}} & {\scriptsize{12472/525.10}}\tabularnewline

{\scriptsize{(500, 1000)}} & {\scriptsize{$(10^{-3},10^{3})$}} & \textbf{\scriptsize{1868/84.54}} & {\scriptsize{62858/3063.96}} & {\scriptsize{62858/3183.97}} & {\scriptsize{18291/768.65}} & {\scriptsize{16425/685.00}}\tabularnewline

{\scriptsize{(500, 1000)}} & {\scriptsize{$(10^{-3},10^{3})$}} & \textbf{\scriptsize{6038/276.19}} & {\scriptsize{*/2.34e-10}} & {\scriptsize{*/3.63e-10}} & {\scriptsize{18331/776.90}} & {\scriptsize{15297/641.28}}\tabularnewline

\hline

{\scriptsize{(1000, 2000)}} & {\scriptsize{$(10^{-2},10^{3})$}} & \textbf{\scriptsize{1132/210.81}} & {\scriptsize{*/1.06e-12}} & {\scriptsize{*/9.77e-13}} & {\scriptsize{23314/3974.38}} & {\scriptsize{19020/3237.59}}
\tabularnewline

{\scriptsize{(1000, 2000)}} & {\scriptsize{$(10^{-2},10^{3})$}} & \textbf{\scriptsize{3848/722.90}} & {\scriptsize{*/1.45e-08}} & {\scriptsize{*/1.50e-08}} & {\scriptsize{24491/4187.51}} &{\scriptsize{18788/3199.92}}
\tabularnewline

{\scriptsize{(1000, 2000)}} & {\scriptsize{$(10^{-1},10^{4})$}} & \textbf{\scriptsize{1200/222.01}} & {\scriptsize{*/1.23e-12}} & {\scriptsize{*/8.72e-13}} &  {\scriptsize{23293/3969.59}} & {\scriptsize{19033/3239.94}}
\tabularnewline

{\scriptsize{(1000, 2000)}} & {\scriptsize{$(10^{-1},10^{4})$}} & \textbf{\scriptsize{4181/786.05}} & {\scriptsize{*/2.09e-08}} & {\scriptsize{*/2.32e-08}} & {\scriptsize{24503/4173.13}} & {\scriptsize{18781/3196.47}}
\tabularnewline

{\scriptsize{(1000, 2000)}} & {\scriptsize{$(10^{-1},10^{5})$}} & \textbf{\scriptsize{1104/201.44}} & {\scriptsize{35820/6646.22}} & {\scriptsize{35820/6875.80}} & {\scriptsize{22498/3832.11}} & {\scriptsize{18868/3207.92}}
\tabularnewline

{\scriptsize{(1000, 2000)}} & {\scriptsize{$(10^{-1},10^{5})$}} & \textbf{\scriptsize{3987/739.08}} & {\scriptsize{*/1.88e-08}} & {\scriptsize{*/1.96e-08}} & {\scriptsize{22219/3782.47}}  & {\scriptsize{19405/3300.10}}
\end{tabular}
\par\end{centering}
\caption{\scriptsize Iteration counts and runtimes (in seconds) for the Vector QP
problem with Box Constraints in \S~\ref{subsec:strconvex_qpbox}. The tolerances are set to $10^{-13}$. Entries marked with * did not converge in the time limit of $7200$ seconds. The ATR metric is 7.84.\label{tab1013:qpbox}}
\end{table}

\subsection{Comments about the numerical results}
As seen from Tables~\ref{tab108:logL1}, \ref{tab1013:Lasso}, \ref{tab1013:Lasso-2}, \ref{tab108:qpsimplex}, \ref{tab1013:qpsimplex}, \ref{tab108:qpbox}, and \ref{tab1013:qpbox}, RPF-SFISTA was the most efficient method and the only method able to find a solution of the desired accuracy (either $10^{-8}$ or $10^{-13}$) within the time limit of 2 hours on every problem instance considered. Out of the 12 sparse logistic regression problems instances considered in \S~\ref{logistic regression}, both FISTA-BT and FISTA-R only finished within the time limit on 5 and 7 instances, respectively. Meanwhile, RA-FISTA and GR-FISTA finished on 4 and 7 instances, respectively. It can be seen from Table~\ref{tab108:logL1} that GR-FISTA was the second-best performing method while RPF-SFISTA was the best performing method for this problem class. As indicated by the ATR metric, RPF-SFISTA was on average over 14 times faster than GR-FISTA across all 12 instances.

Out of the 60 Lasso problem instances considered in \S~\ref{Lasso}, both FISTA-BT and FISTA-R finished within the time limit on 52 and 51 instances, respectively. RA-FISTA and GR-FISTA finished within the time limit of 2 hours on 57 and 59 instances, respectively. It can be seen from Tables~\ref{tab1013:Lasso} and \ref{tab1013:Lasso-2} that GR-FISTA was the second-best performing method while RPF-SFISTA was the best performing method for this problem class. As indicated by the ATR metrics in both tables, RPF-SFISTA was approximately 3.87 times faster than GR-FISTA on the first 30 Lasso instances considered and 6.32 times faster on the last 30 instances considered. Hence, RPF-SFISTA was roughly 5 times faster than GR-FISTA across all 60 instances considered.

For the simplex-constrained dense QP instances considered in \S~\ref{subsec:nonconvex_qpsimplex}, both FISTA-BT and FISTA-R finished with a solution of accuracy $10^{-8}$ within the time limit on 11 of the 12 instances while both RA-FISTA and GR-FISTA finished within the time limit on 6 instances. For this accuracy, FISTA-R was the second-best performing method on this problem class while RPF-SFISTA was the best performing method. As indicated by the ATR metric, RPF-SFISTA was on average over 3.3 times faster than FISTA-R across the 12 instances considered. When the desired accuracy was $10^{-13}$, both FISTA-BT and FISTA-R finished within the time limit on 7 of the 12 instances while both RA-FISTA and GR-FISTA finished on 6 instances. For this accuracy, FISTA-R was the second-best performing method while RPF-SFISTA was the best performing method. As indicated by the ATR metric, RPF-SFISTA was on average over 4.6 times faster than FISTA-R across the 12 instances, showing that the gap between RPF-SFISTA and the other codes increased when a higher accuracy of $10^{-13}$ was required. 

For the box-constrained dense QP instances considered in \S~\ref{subsec:strconvex_qpbox}, both FISTA-BT and FISTA-R finished with a solution of accuracy $10^{-8}$ within the time limit of 2 hours on 9 of the 12 instances while both RA-FISTA and GR-FISTA finished on all 12 instances. For this accuracy, GR-FISTA was the second-best performing method on this problem class while RPF-SFISTA was the best performing method. As indicated by the ATR metric, RPF-SFISTA was on average over 7.08 times faster than GR-FISTA across the 12 instances considered. When the desired accuracy was $10^{-13}$, both FISTA-BT and FISTA-R finished within the time limit on just 3 of the 12 instances while both RA-FISTA and GR-FISTA again finished on all 12 instances. GR-FISTA was the second-best performing method while RPF-SFISTA was the best performing method. As indicated by the ATR metric, RPF-SFISTA was on average over 7.8 times faster than GR-FISTA across the 12 instances.

\begin{appendices}
\section{Proof of Proposition~\ref{prop:nest_complex1}}\label{SC-FISTA Appendix Proof}
This section is dedicated to proving Proposition~\ref{prop:nest_complex1}. It is broken up into two subsections. The first subsection is dedicated to proving Proposition~\ref{prop:nest_complex1}(a)-(b) while the second one is dedicated to proving Proposition~\ref{prop:nest_complex1}(c).

\subsection{Proof of Proposition~\ref{prop:nest_complex1}(a)-(b)}\label{Part A proof}
The following lemma establishes a key bound on the distance between the iterate $\xi_j$ and the initial point $z_{l-1}$ of the $l$-th cycle.
\begin{lemma}
For every iteration index $j\geq 1$ generated during the l-th cycle of RPF-SFISTA, it holds that
\begin{equation}\label{Bounded W}
    \|\xi_j-z_{l-1}\|^2 \leq C_{\bar \mu}(z_{l-1})
\end{equation}
where $C_{\bar \mu}(\cdot)$ is as in \eqref{D0 def} and $z_{l-1}$ is the initial point of the $l$-th cycle. 
\end{lemma}
\begin{proof}
    Let $j \geq 1$ be an iteration index generated during the $l$-th cycle of RPF-SFISTA. It follows immediately from the fact that $\phi$ is $\bar \mu$-convex and relation \eqref{ineq:nu-convex} with $\Psi=\phi$ and $\nu=\bar \mu$ that the following relations hold
    \begin{equation}\label{Key Bound on X star}
    \|\xi_j-z^{*}\|^2\leq \frac{2}{\bar \mu}\left[\phi(\xi_j)-\phi(z^{*})\right], \quad \|z_{l-1}-z^{*}\|^2\leq \frac{2}{\bar \mu}\left[\phi(z_{l-1})-\phi(z^{*})\right].
    \end{equation}
The above relations, triangle inequality, and relation \eqref{function value decrease XI} then imply that
\begin{align*}
    \|\xi_j-z_{l-1}\|^2&\leq 2\|\xi_j-z^{*}\|^2+2\|z^{*}-z_{l-1}\|^2 \nonumber\\
    &\overset{\eqref{Key Bound on X star}}{\leq} \frac{4}{\bar \mu}\left[\phi(\xi_j)-\phi(z^{*})\right]+\frac{4}{\bar \mu}\left[\phi(z_{l-1})-\phi(z^{*})\right]\overset{\eqref{function value decrease XI}}{\leq} \frac{8}{\bar \mu}\left[\phi(z_{l-1})-\phi(z^{*})\right].
\end{align*}
The conclusion of the lemma then immediately follows from the above relation and the definition of $C_{\bar \mu}(\cdot)$ in \eqref{D0 def}.
\end{proof}

The proof of Proposition~\ref{prop:nest_complex1}(a)-(b) is now presented. The proof of part (a) has a similar pattern to the proof of Proposition A.5 in \cite{SujananiMonteiro}, but we include the proof here for the sake of completeness.
\begin{proof}[Proof of Proposition~\ref{prop:nest_complex1}(a)-(b)]
(a) Consider the $l$-th cycle of RPF-SFISTA and let $m$ denote the first quantity in \eqref{eq:eq1}. Using 
this definition and the
inequality $\log (1+ \alpha) \ge \alpha/(1+\alpha)$ for
any $\alpha>-1$, it can easily be seen that
\begin{equation}\label{Q bound-2}
\left(1+Q_l^{-1}\right)^{2(m-1)} \ge 
\frac{C_{\bar \mu}(z_{l-1})\zeta_l^2}{\chi \hat \epsilon^2}.
\end{equation}
We claim that the $l$-th cycle of RPF-SFISTA terminates in either step 4 or step 5 in at most $m$ iterations.
% Indeed, assume for contradiction that
It is thus sufficient to show that
if the $l$-th cycle of RPF-SFISTA has not stopped in step 4
up to and including the
$m$-th iteration, then
it must stop successfully in step 5
at the $m$-th iteration.
So, assume that the $l$-th cycle of
RPF-SFISTA has not stopped in step 4
up to the
$m$-th iteration.
It is easy to see then in view of step~4 of RPF-SFISTA that relation
\eqref{restart condition} holds with $j=m$.

It then follows from this relation, inequality \eqref{ubound-1} with $j=m$, the fact that $x_0$ is set as $z_{l-1}$ at the beginning of the $l$-th cycle, relations \eqref{Bounded W} and \eqref{Q bound-2}, and inequality \eqref{sumAkbound} with $j=m$ that
\begin{align}\label{bound diff-1-2}
C_{\bar \mu}(z_{l-1})&\overset{\eqref{Bounded W}}{\geq}\|\xi_{m}-z_{l-1}\|^{2}=\|\xi_{m}-x_0\|^{2} \overset{\eqref{restart condition}}{\geq} \chi A_{m}L_{m} \|y_{m}-\tilde x_{m-1}\|^2
\overset{\eqref{ubound-1}}{\ge}
\frac{\chi}{\zeta_l^2}A_mL_m \|v_{m}\|^2\\
&\overset{\eqref{sumAkbound}}{\geq}\frac{\chi}{\zeta_l^2}
\left(1+Q_l^{-1} \right)^{2(m-1)}
\|v_m\|^2 \overset{\eqref{Q bound-2}}{\geq} \frac{C_{\bar \mu}(z_{l-1})}{\hat \epsilon^2}\|v_m\|^2
\end{align}
which implies that the termination criterion
\eqref{u sigma criteria} in step 5 of RPF-SFISTA is satisfied.
Hence, the $l$-th cycle of RPF-SFISTA
must
successfully stop at the
end of its $m$-th iteration and
the claim thus holds.
Moreover, it is easy to see from relation \eqref{upper bound} that the second quantity in 
\eqref{eq:eq1} is a bound
on the total number of times that step 2 of RPF-SFISTA needs to be repeated. Since every time step 2 of RPF-SFISTA is performed only one resolvent evaluation is needed, the conclusion of part (a) follows.

(b) To show part (b), assume that the $l$-th cycle of RPF-SFISTA terminates in its step 5 and outputs a quadruple $(y,v,\xi,L)$. It then follows that $(y,v,\xi,L)=(y_j,v_j,\xi_j,L_j)$ where $j$ is an iteration index generated during the $l$-th cycle of RPF-SFISTA. The inclusion in \eqref{ubound-1} and the termination criterion \eqref{u sigma criteria} in step 5 of RPF-SFISTA then immediately imply that pair $(y,v)=(y_j,v_j)$ satisfies \eqref{acg problem} and hence is an $\hat \epsilon$-optimal solution of \eqref{OptProb1}. Moreover, it follows from the fact that $L=L_{j}$ for an iteration index $j$ generated during the $l$-th cycle, both inequalities in \eqref{upper bound}, the fact that the way $\underbar M_l$ is chosen in step 0 implies that $\underbar M_l \geq \bar M_0$, and relation \eqref{Bounding Upper Curvature} that
\[\bar M_0\leq \underbar M_l\leq L \overset{\eqref{upper bound}}{\leq} \max\{\underbar M_l,\kappa \bar L\} \overset{\eqref{Bounding Upper Curvature}}{\leq} \max\left\{\bar M_0,\kappa \bar L\right\}\]
which implies that $L$ satisfies the second relation in \eqref{Func Value Decrease}. To see the first relation in \eqref{Func Value Decrease}, observe that the first conclusion of Lemma~\ref{First Lemma} and the facts that $\xi=\xi_j$ and $y=y_j$ imply that $\phi(\xi)\leq \phi(y)$. It also follows from combining relations \eqref{function value decrease XI} and \eqref{Key Func Bound} that $\phi(\xi)\leq \phi(z_0)$, which together with the above relation implies that the first relation in \eqref{Func Value Decrease} holds. Proposition~\ref{prop:nest_complex1}(b) then immediately follows from the above conclusions.\end{proof}

\subsection{Proof of Proposition~\ref{prop:nest_complex1}(c)}\label{Part B Proof}

% Some remarks about Propositions~\ref{prop:nest_complex1} and \ref{strongly convex fista} are now given. A key difference between the ADAP-FISTA method developed in \cite{SujananiMonteiro} and SC-FISTA is that ADAP-FISTA is only guaranteed to terminate successfully if $\mu \in (0,\hat \mu]$ where $\hat \mu$ is the strong convexity parameter of $f$, which is the smooth part of the composite function. On the other hand, Proposition~\ref{strongly convex fista} shows that SC-FISTA is guaranteed to terminate if $\mu \in (0,\bar \mu]$ where $\hat \mu$ is the strong convexity parameter of the entire composite function, $\phi$. To prove SC-FISTA's complexity bound under this assumption, a novel estimate sequence $\gamma_j$ which lowers bounds $\psi$ is constructed.

For the remainder of this subsection, consider the $l$-th cycle of RPF-SFISTA and assume that $j\geq 1$ is an iteration index generated during the cycle and that $\mu=\mu_{l-1}$ is in the interval $(0,\bar \mu]$. 

The main novelty in the proof of Proposition~\ref{prop:nest_complex1}(c) lies in our careful construction of a sequence $\gamma_{j}$ which lower bounds the entire composite function $\phi$. This construction is thus essential to establishing that RPF-SFISTA is $\bar \mu$-universal as opposed to just $\bar \mu_f$-universal.

% proving that a cycle of RPF-SFISTA always terminates successfully if $\mu$ is in the interval $(0,\bar \mu]$ as opposed to the more restrictive interval $(0,\bar \mu_f)$ if were to assume that ??. Hence, 

The proofs of Lemmas~\ref{lem:gamma-sfista} and \ref{lm:hysub-3-sfista} below follow closely to the proofs of Lemmas A.6 and A.7 in \cite{SujananiMonteiro} except for minor modifications due to our novel construction of the sequence $\gamma_j$. We include the proofs here for completeness.
The first lemma establishes key properties of the iterates of RPF-SFISTA.

\begin{lemma}\label{lem:gamma-sfista}
	For every $ j\ge 1 $ and $x \in \mathbb E$, define 
	\begin{align}
		\gamma_j(x)&:=\phi(y_{j}) + 2\left[\ell_{f}(y_{j},\tilde x_{j-1})-f(y_{j})\right] + \inner{s_{j}}{x - y_{j}}
		+ \frac{\mu}4 \|x-y_{j}\|^2, \label{def:gamma-sfista}
	\end{align}
where $\phi:=f+h$ and $s_{j}$ are as in \eqref{OptProb1} and \eqref{eq:sk}, respectively.
Then, for every $j \ge 1$, we have:
	\begin{align}
		y_{j} &= \argmin _{x}\left\{\gamma_{j}(x)+\frac{L_{j}}{2}\left\|x-\tilde{x}_{j-1}\right\|^{2}\right\};\label{eq:min-sfista}\\
		x_{j}&=\underset{x}\argmin\left\{a_{j-1} \gamma_{j}(x)+\tau_{j-1} \left\|x-x_{j-1}\right\|^{2} /2 \right\}.\label{xgamma}
		\end{align}
\end{lemma}
\begin{proof}
It follows from the definition of $s_j$ in \eqref{eq:sk} and the fact that $\nabla \gamma_j(y_{j})=s_{j}$ that $y_{j}$ satisfies the optimality condition for \eqref{eq:min-sfista}. This observation implies that relation \eqref{eq:min-sfista} must hold. Moreover, it follows from relations \eqref{eq:taunext-sfista1} and \eqref{eq:xnext-sfista1} that
\begin{align*}
 a_{j-1} \nabla \gamma_j(x_{j})+\tau_{j-1}(x_{j}-x_{j-1})&=a_{j-1}s_{j}+\frac{a_{j-1}\mu}{2}(x_{j}-y_{j})+\tau_{j-1}(x_{j}-x_{j-1})\\
 &\overset{\eqref{eq:taunext-sfista1}}{=} a_{j-1} s_{j} -\frac{\mu a_{j-1}}{2}y_{j}-\tau_{j-1}x_{j-1}+ \tau_{j}x_{j}
 \overset{\eqref{eq:xnext-sfista1}}{=}0
\end{align*}
which implies that relation \eqref{xgamma} holds. 
\end{proof}

\vgap

\begin{lemma} \label{lm:hysub-3-sfista}
	For every $j\ge 1$ and $x \in \mathbb E$, we have
	\begin{align}\label{ineq:recur}
	A_{j-1}\gamma_j(y_{j-1}) &+ a_{j-1}\gamma_j(x) + \frac{\tau_{j-1}}2 \|x_{j-1} - x\|^2 - \frac{\tau_{j}}2 \|x_{j} - x\|^2 \nonumber \\
	&\ge A_{j}\phi(y_{j})+\frac{\chi A_{j} L_{j}}{2}\|y_{j}-\tilde x_{j-1}\|^2.  
	\end{align}
\end{lemma}

\begin{proof}
Relation \eqref{xgamma}, the second relation in \eqref{eq:taunext-sfista1}, the fact that $\Psi_j:=a_{j-1}\gamma_j(\cdot)+\tau_{j-1} \|\cdot-x_{j-1}\|^2/2$ is  $(\tau_{j-1}+\mu a_{j-1}/2)$-convex, and relation
\eqref{ineq:nu-convex} with $\Psi=\Psi_j$ and $\nu=\tau_{j}$ imply that 
\begin{align}\label{First Relation}
a_{j-1}\gamma_j(x) + \frac{\tau_{j-1}}2 \|x-x_{j-1}\|^2 - \frac{\tau_{j}}2 \|x-x_{j}\|^2 \ge a_{j-1}\gamma_j(x_{j}) + \frac{\tau_{j-1}}2 \|x_{j}-x_{j-1}\|^2 \quad
\forall x \in \mathbb E.
\end{align}
	It then follows from the convexity of $\gamma_j$, the definitions of $\tx_{j-1}$ and $A_{j}$ in \eqref{def:ak-sfista1} and \eqref{eq:taunext-sfista1}, respectively, and the second equality in \eqref{tauproperty}, that
	\begin{align}\label{Second Relation}
		A_{j-1}\gamma_j(y_{j-1}) &+ a_{j-1}\gamma_j(x_{j}) + \frac{\tau_{j-1}}2 \|x_{j}-x_{j-1}\|^2 \nonumber\\
		&\ge A_{j} \gamma_j\left( \frac{A_{j-1}y_{j-1}+a_{j-1}x_{j}}{A_{j}} \right) + \frac{\tau_{j-1}A^2_{j}}{2a_{j-1}^2}\left\| \frac{A_{j-1}y_{j-1}+a_{j-1}x_{j}}{A_{j}}- \frac{A_{j-1}y_{j-1}+a_{j-1}x_{j-1}}{A_{j}} \right\| ^2 \nonumber\\
	&\overset{\eqref{def:ak-sfista1}}{\geq} A_{j} \min_{x}\left[ \gamma_j\left( x \right) + \frac{\tau_{j-1}A_{j}}{2a_{j-1}^2} \left\| x-\tx_{j-1}\right\| ^2\right] \nonumber \\
		&\overset{\eqref{tauproperty}}{=} A_{j}\min_{x}\left\lbrace \gamma_j(x) + \frac{L_{j}}{2}\|x-\tx_{j-1}\|^2\right\rbrace \nonumber\\
		&\overset{\eqref{eq:min-sfista}}{=} A_{j}\left[ \gamma_j(y_{j}) + \frac{L_{j}}{2}\|y_{j}-\tx_{j-1}\|^2\right] \nonumber\\
		&\overset{\eqref{def:gamma-sfista}}{=} A_{j}\left[ \phi(y_{j}) + 2 [\ell_{f}(y_{j},\tilde x_{j-1})-f(y_{j}) ] +
	\frac{L_{j}}{2}\|y_{j}-\tx_{j-1}\|^2 \right]  \nonumber\\
		&\overset{\eqref{ineq check}}{\geq} A_{j}\left[ \phi(y_{j})  +
		\frac{\chi L_{j}}{2}\|y_{j}-\tx_{j-1}\|^2 \right].
	\end{align}
Combining relations \eqref{First Relation} and \eqref{Second Relation} then immediately implies the conclusion of the lemma.
	\end{proof}
	
We now present a technical lemma that is important for proving the lemma that directly follows it. The proof can be found in \cite{LiangMonteiro-DAIPPM18}.
\begin{lemma}\label{inclusion}
	Assume that $\varphi$ is a $\xi$-strongly convex function and let
	$(y,\eta) \in \mathbb E \times \R$ be such that $
	0\in \partial_\eta \varphi(y)$.
	Then, %for any $ u \in \dom \phi$,
	\[
	0 \in \partial_{2\eta}\left(  \varphi(\cdot) - \frac{\xi} 4 \|\cdot-y\|^2 \right) (y).
	\]
\end{lemma}	

The following lemma establishes that the estimate sequence $\gamma_j$ constructed in \eqref{lem:gamma-sfista} lower bounds $\phi$.
\begin{lemma} \label{lm:gammakphi}
For every $j \ge 1$, we have that
$\gamma_j \leq \phi$. 
\end{lemma}

\begin{proof}
For every $j \geq 1$, define
    \begin{equation}\tilde \gamma_j(x):= \ell_{f}(x;\tilde x_{j-1})  + h(x) \label{def:tgamma-sfista}.
    \end{equation}
The fact that $f$ is convex implies that $\tilde \gamma_j \leq \phi$. It follows from the definition of $y_{j}$ in \eqref{eq:ynext-sfista1} that
    \begin{equation}
		y_{j}= \argmin_{x}\left\{\tilde{\gamma}_{j}(x)+\frac{L_{j}}{2}\left\|x-\tilde{x}_{j-1}\right\|^{2}\right\} \label{eq:min-sfista'}.
		\end{equation}
Now, it is easy to see from definition of $s_{j}$ in \eqref{eq:sk} and relation \eqref{eq:min-sfista'} that $s_{j} \in \partial \tilde \gamma_j(y_{j})$. Thus, by the subgradient inequality and the fact that $\tilde \gamma_j(x) \leq \phi(x)$, we have that for all $x \in \mathbb E$:
    \[\phi(x)\geq \tilde \gamma_j(x) \geq \tilde \gamma_j(y_{j})+\inner{s_{j}}{x-y_{j}}=\phi(y_{j})+\inner{s_{j}}{x-y_{j}}-\eta_{j},        \]
    where $\eta_{j}:=\phi(y_{j})-\tilde \gamma(y_{j})$. Hence, by the definition of $\eta_{j}$-subdifferential, it follows that $s_{j}\in \partial_{\eta_{j}}\phi(y_{j})$.
    Now, note that $\phi$ is a $\mu$-strongly convex function since $\mu=\mu_{l-1}$ is assumed to be in the interval $(0,\bar \mu]$. Thus, by Lemma~\ref{inclusion} with $\xi=\mu$ and $\varphi(\cdot)=\phi(\cdot)-\inner{s_{j}}{\cdot-y_{j}}$,  we have 
    \begin{equation}\label{2 eta sub}
    s_{j} \in \partial_{2\eta_{j}}\left(\phi(\cdot)-\frac{\mu}{4}\|\cdot-y_{j}\|^2\right)(y_{j}).\\
    \end{equation}
 Hence, it follows from the definition of $\eta_{j}$ above, relation \eqref{2 eta sub}, the fact that $h=\phi-f$, and the definitions of $\tilde \gamma_j(\cdot)$ and $\gamma_j(\cdot)$ in \eqref{def:tgamma-sfista} and \eqref{def:gamma-sfista}, respectively that
    \begin{align*}
    \phi(x)&\overset{\eqref{2 eta sub}}{\geq}\phi(y_{j})+\inner{s_{j}}{x-y_{j}}+\frac{\mu}{4}\|x-y_{j}\|^2-2\eta_{j}\\
    &=2\tilde \gamma_j(y_{j})-\phi(y_{j}) + \inner{s_{j}}{x - y_{j}}
		+ \frac{\mu}4 \|x-y_{j}\|^2\\
  &\overset{\eqref{def:tgamma-sfista}}{=}2\ell_{f}(y_{j};\tilde x_{j-1}) + 2h(y_{j})-\phi(y_{j}) + \inner{s_{j}}{x - y_{j}}
		+ \frac{\mu}4 \|x-y_{j}\|^2\\
  &=2\ell_{f}(y_{j};\tilde x_{j-1}) + 2[\phi(y_{j})-f(y_{j})]-\phi(y_{j}) + \inner{s_{j}}{x - y_{j}}
		+ \frac{\mu}4 \|x-y_{j}\|^2
  \overset{\eqref{def:gamma-sfista}}{=}\gamma_j(x),
    \end{align*}
from which statement of the lemma immediately follows.     
\end{proof}

\begin{lemma} \label{lm:hysub-4-fista}
For every $j \ge 1$ and $x \in \mathcal H$, we have
\[
\sigma_{j-1}(x) - \sigma_{j} (x) \ge   \frac{\chi A_{j}L_{j}}{2} \|y_{j} - \tx_{j-1} \|^2
\]
where
\[
\sigma_j(x) := A_j [ \phi(y_j) - \phi(x) ] + \frac{\tau_j}{2} \|x-x_j\|^2.
\]
\end{lemma}

\begin{proof}
Using Lemma \ref{lm:hysub-3-sfista} and Lemma~\ref{lm:gammakphi}  we have
\begin{align*}
A_{j-1} \phi (y_{j-1}) + & a_{j-1}\phi (x) + \frac{\tau_{j-1}}2 \|x_{j-1} - x\|^2 - \frac{\tau_{j}}2 \|x_{j} - x\|^2 \\
	& \ge  A_{j} \phi(y_{j}) + \frac{\chi A_{j}L_{j}}{2} \| y_{j} - \tx_{j-1} \|^2 .
\end{align*}
The conclusion of the lemma now follows by subtracting $A_{j} \phi(x)$ from both sides of the above inequality, and using
the first equality in \eqref{eq:taunext-sfista1} and the definition of $\sigma_j(x)$.
\end{proof}

% ---------????????----------

% Choose $\mu_i >0$
% such that
% \[
% \mu_i \|x_0-y_i\|^2 \le
% \frac{\varepsilon}4
% \]
% Then
% \begin{align*}
%    \frac{1}{A_j} \sum_i a_i \mu_i \|y_j-y_i\|^2
% &\le \frac{2}{A_j} \sum_i a_i \mu_i \left[ \|y_j-x_0\|^2 + \|x_0-y_i\|^2 \right] \\
% & \le
% \frac{\varepsilon}2 +
% \frac{2}{A_j} \sum_i a_i \mu_i  \|y_j-x_0\|^2
% \end{align*}

The following result is important for proving Proposition~\ref{prop:nest_complex1}(c)
\begin{lemma} \label{lm:hysub-5-ac}
For every $j \ge 2$ and $x \in \mathcal H$, it holds that
\begin{align}\label{key convergence inequality 2}
 A_{j-1} [\phi(\xi_{j-1}) - \phi(x) ] + \frac{\tau_{j-1}}2 \|x-x_{j-1}\|^2 \le
 \frac12 \|x-x_0\|^2 - \frac{\chi}{2} \sum_{i=1}^{j-1} A_{i}L_{i} \| y_{i} - \tx_{i-1} \|^2.
\end{align}
\end{lemma}

% -----------
% \[
% \|x-z^{*}\|^2 \le \frac{2(\phi(x) - \phi^*)}{\mu_\phi}
% \]
% \[
% \|y_j-x_0\|^2 \le 2 \|x_0 -z^{*}\|^2 +2 \|y_j-z^{*}\|^2  \le
% 2 \|x_0 -z^{*}\|^2 +
% \frac{4(\phi(y_j) - \phi^*)}{\mu_\phi}
% \]

% -----------

\begin{proof}
It follows from summing the inequality of Lemma \ref{lm:hysub-4-fista} from $j=1$ to $j=j-1$, using the facts that $A_0=0$ and $\tau_0=1$, and using the definition of $\sigma_j(\cdot)$ in Lemma~\ref{lm:hysub-4-fista} that
\begin{align*}
A_{j-1} [\phi(y_{j-1}) - \phi(x) ] + \frac{\tau_{j-1}}2 \|x-x_{j-1}\|^2 \le
 \frac12 \|x-x_0\|^2 - \frac{\chi}{2} \sum_{i=1}^{j-1} A_{i}L_{i} \| y_{i} - \tx_{i-1} \|^2.
\end{align*}
Relation \eqref{key convergence inequality 2} then immediately follows from the above relation and the fact that the first conclusion of Lemma~\ref{First Lemma} implies that $\phi(\xi_{j-1})\leq \phi(y_{j-1})$.
\end{proof}

We are now ready to prove Proposition~\ref{prop:nest_complex1}(c).
\begin{proof}[Proof of Proposition~\ref{prop:nest_complex1}(c)]
Consider the $l$-th cycle of RPF-SFISTA and assume that it is performed with $\mu_{l-1} \in (0,\bar \mu]$. Using relation \eqref{key convergence inequality 2} with $x=\xi_{j-1}$, it follows that
\[
\|\xi_{j-1}-x_0\|^2 \overset{\eqref{key convergence inequality 2}}{\geq}\chi \sum_{i=1}^{j-1}A_{i}L_{i} \|y_{i}-\tilde x_{i-1}\|^2 \geq \chi A_{j-1}L_{j-1}\|y_{j-1}-\tilde x_{j-2}\|^2.
\]
It then follows from the above relation that for any iteration index $j\geq 1$ generated during the $l$-th cycle, it holds that
\begin{equation}\label{bound diff-3}
\|\xi_{j}-x_0\|^2 \geq \chi A_{j}L_{j}\|y_{j}-\tilde x_{j-1}\|^2.
\end{equation}
Hence, relation \eqref{bound diff-3} implies that the inequality \eqref{restart condition} checked in step 4 of RPF-SFISTA always holds for every iteration index $j\geq 1$ generated during the $l$-th cycle
and hence the $l$-th cycle of RPF-SFISTA never terminates in step 4. This observation together with Proposition~\ref{prop:nest_complex1}(a)-(b) then immediately imply that the $l$-th cycle must terminate successfully in its step 5 with a quadruple $(y,v,\xi,L)$ that satisfies \eqref{Func Value Decrease} and such that $(y,v)$ is an $\hat \epsilon$-optimal solution of \eqref{OptProb1} in at most \eqref{eq:eq1} ACG iterations/resolvent evaluations.
\end{proof}

\end{appendices}

% \renewcommand{\baselinestretch}{.5}
%\tiny
\scriptsize
%\footnotesize

\typeout{}
\bibliographystyle{plain}
\bibliography{Proxacc_ref}

\end{document}